\documentclass[12pt]{article}
\usepackage{amssymb,amsmath,amsthm, amsfonts}
\usepackage{graphicx}
\usepackage{epsfig}
\usepackage{tikz}
\usepackage{mathrsfs}

\def\disp{\displaystyle}

\def\dref#1{(\ref{#1})}

\theoremstyle{plain}
\newtheorem{theorem}{Theorem}[section]
\newtheorem{lemma}{Lemma}[section]

\theoremstyle{definition}
\newtheorem{definition}{Definition}[section]

\newtheorem{remark}{Remark}[section]

\numberwithin{equation}{section}

\begin{document}

\title{\bf Global weak solutions in a three-dimensional Keller-Segel-Navier-Stokes system with  nonlinear diffusion}

\author{
Jiashan Zheng\thanks{Corresponding author.   E-mail address:
 zhengjiashan2008@163.com (J.Zheng)}
 \\
    School of Mathematics and Statistics Science,\\
     Ludong University, Yantai 264025,  P.R.China \\
}
\date{}


\maketitle \vspace{0.3cm}
\noindent
\begin{abstract}
The coupled quasilinear Keller-Segel-Navier-Stokes system
$$
 \left\{
 \begin{array}{l}
   n_t+u\cdot\nabla n=\Delta n^m-\nabla\cdot(n\nabla c),\quad
x\in \Omega, t>0,\\
    c_t+u\cdot\nabla c=\Delta c-c+n,\quad
x\in \Omega, t>0,\\
u_t+\kappa(u \cdot \nabla)u+\nabla P=\Delta u+n\nabla \phi,\quad
x\in \Omega, t>0,\\
\nabla\cdot u=0,\quad
x\in \Omega, t>0\\
 \end{array}\right.\eqno(KSNF)
 $$
is considered under Neumann boundary conditions for $n$ and $c$ and no-slip boundary conditions for
$u$ in three-dimensional bounded domains $\Omega\subseteq \mathbb{R}^3$ with smooth boundary, where $m>0,\kappa\in \mathbb{R}$ are given
constants,
$\phi\in W^{1,\infty}(\Omega)$.
If
$ m> 2$,
then for all reasonably regular initial data, a corresponding initial-boundary value
problem for $(KSNF)$ possesses a globally defined weak solution.
\end{abstract}

\vspace{0.3cm}


\newpage
\section{Introduction}

 Chemotaxis is a biological process in which cells move toward a chemically more favorable environment (see Hillen and Painter \cite{Hillen}).
   In 1970, Keller and
Segel (see Keller and Segel \cite{Keller2710,Keller79})
proposed a mathematical model for chemotaxis phenomena through a
system of parabolic equations
%
%
%
%
%
  (see e.g. Winkler et al. \cite{Bellomo1216,Horstmann791,Winkler793}, Osaki and Yagi \cite{Osaki3drr}, Horstmann
\cite{Horstmann2710}).
To describe chemotaxis of cell populations,  the signal is  produced
by the cells,  an important variant of the quasilinear
chemotaxis model
\begin{equation}
 \left\{\begin{array}{ll}
 n_t=\nabla\cdot( D(n)\nabla n)-\chi\nabla\cdot( S(n)\nabla c),\\
 \disp{ c_t=\Delta c- c +n}
 \end{array}\right.\label{722344101.2x16677}
\end{equation}
 was initially proposed by Painter and Hillen (\cite{Painter55677}, see also Winkler et al. \cite{Bellomo1216,Tao793}),
%
%
 where $n$
denotes the cell density and $c$ describes the concentration of the chemical signal
secreted by cells.
 The function $ S$ measures the chemotactic sensitivity, which may depend on $n,$ $ D(n)$
 is the diffusion function.
The results about the chemotaxis model \dref{722344101.2x16677} appear to be rather complete, which dealt with the  problem \dref{722344101.2x16677}
whether the solutions are global bounded or blow-up
(see  Cie\'{s}lak et al. \cite{Cie794,Cie791,Cie72},   Hillen \cite{Hillen}, Horstmann et al. \cite{Horstmann79}, Ishida et al. \cite{Ishida},  Kowalczyk \cite {Kowalczyk7101},
 Winkler et al. \cite{Tao794,Winkler72,Winkler793}). In fact,
 Tao and Winkler (\cite{Tao794}), 
 proved   that the solutions of \dref{722344101.2x16677}  are global and bounded provided that
  $\disp\frac{ S(n)}{ D(n)}\leq c(n+1)^{\frac{2}{N}+\varepsilon}$ for
all $n\geq0$ with some $\varepsilon > 0$ and $c > 0$,
and  $  D(n)$ satisfies some another technical conditions.
 For the more related works in this direction, we mention that a corresponding quasilinear
version, the logistic damping or the signal is  consumed
by the cells
has been deeply investigated by  Cie\'{s}lak and Stinner \cite{Cie791,Cie201712791}, Tao and Winkler \cite{Tao794,Winkler79,Winkler72} and Zheng et al. \cite{Zheng00,Zheng33312186,Zhengssssdefr23,Zhengssssssdefr23}.

In various situations, however, the migration of bacteria is furthermore substantially affected by
changes in their environment (see Winkler et al. \cite{Bellomo1216,Tao41215}).
As in the quasilinear
Keller-Segel system \dref{722344101.2x16677} where the chemoattractant is produced by cells, the corresponding  chemotaxis--fluid model is then
 is then quasilinear  Keller-Segel-Navier-Stokes system of the form
 \begin{equation}
 \left\{\begin{array}{ll}
   n_t+u\cdot\nabla n=\nabla\cdot(  D(n)\nabla n)-\nabla\cdot(  S(n)\nabla c),\quad
x\in \Omega, t>0,\\
    c_t+u\cdot\nabla c=\Delta c-c+n,\quad
x\in \Omega, t>0,\\
u_t+\kappa (u\cdot\nabla)u+\nabla P=\Delta u+n\nabla \phi,\quad
x\in \Omega, t>0,\\
\nabla\cdot u=0,\quad
x\in \Omega, t>0,\\
 \end{array}\right.\label{1ssxdcfvgb.1}
\end{equation}
where   $n$ and $c$ are denoted as before, $u$ and $P$ stand for the velocity of incompressible
fluid and the associated pressure, respectively. $\phi$  is a given potential function and $\kappa\in \mathbb{R}$ denotes the
strength of nonlinear fluid convection.
Problem  \dref{1ssxdcfvgb.1} is proposed to describe
chemotaxis--fluid interaction in cases when the evolution of the chemoattractant is essentially dominated
by production through cells (\cite{Bellomo1216,Hillen}).

 If the signal is consumed, rather than produced,
by the cells, Tuval et al. (\cite{Tuval1215}) proposed the following model
 \begin{equation}
 \left\{\begin{array}{ll}
   n_t+u\cdot\nabla n=\nabla\cdot(  D(n)\nabla n)-\nabla\cdot( nS(c)\nabla c),\quad
x\in \Omega, t>0,\\
    c_t+u\cdot\nabla c=\Delta c-nf(c),\quad
x\in \Omega, t>0,\\
u_t+\kappa (u\cdot\nabla)u+\nabla P=\Delta u+n\nabla \phi,\quad
x\in \Omega, t>0,\\
\nabla\cdot u=0,\quad
x\in \Omega, t>0.\\
 \end{array}\right.\label{1.1hhjjddssggtyy}
\end{equation}
Here  $f(c)$ is the consumption rate of the oxygen by the cells.
Approaches based on a natural energy functional, the (quasilinear) chemotaxis-(Navier-)Stokes system \dref{1.1hhjjddssggtyy} has been studied in the last few years and the main focus is on
the solvability result (see e.g.
Chae,  Kang and  Lee \cite{Chaexdd12176},
Duan, Lorz, Markowich \cite{Duan12186},
Liu and Lorz  \cite{Liu1215,Lorz1215},
 Tao and Winkler   \cite{Tao41215,Winkler31215,Winkler61215,Winkler51215}, Zhang and Zheng \cite{Zhang12176} and references therein).
 For instance, if $\kappa=0$ in \dref{1.1hhjjddssggtyy},  the model is simplified to the chemotaxis-Stokes equation.
 In \cite{Winkler21215},  Winkler showed the  global weak
solutions of \dref{1.1hhjjddssggtyy}  in bounded three-dimensional domains. Other variants of the model of \dref{1.1hhjjddssggtyy} that include  porous
medium-type diffusion and $S$ being a chemotactic sensitivity tensor, one can see Winkler (\cite{Winkler11215}) and Zheng (\cite{Zhengssssssssdefr23})
 and the references therein for details.

In contrast to problem \dref{1.1hhjjddssggtyy}, the mathematical analysis of the Keller-Segel-Stokes system \dref{1ssxdcfvgb.1} ($\kappa=0$) is quite few (Black \cite{Blackffg},
Wang et al.
\cite{Li33223321215,Wang445521215,Wangss21215}).
Among these results, Wang et al. (\cite{Wang445521215,Wangss21215})
proved  the global boundedness of solutions to the two-dimensional  and there-dimensional Keller-Segel-Stokes system \dref{1ssxdcfvgb.1} when $S$ is a tensor satisfying some dampening condition with respective to $n$.  However, for the there-dimensional fully Keller-Segel-Navier-Stokes system \dref{1ssxdcfvgb.1} ($\kappa\in \mathbb{R}$), to the best our knowledge, there is no result on global
solvability. Motivated by the above works, we will investigate the interaction of the fully quasilinear Keller-Segel-Navier-Stokes in this paper. Precisely, we shall consider the following initial-boundary
problem
\begin{equation}
 \left\{\begin{array}{ll}
   n_t+u\cdot\nabla n=\Delta n^m-\nabla\cdot(n\nabla c),\quad
x\in \Omega, t>0,\\
    c_t+u\cdot\nabla c=\Delta c-c+n,\quad
x\in \Omega, t>0,\\
u_t+\kappa (u\cdot\nabla)u+\nabla P=\Delta u+n\nabla \phi,\quad
x\in \Omega, t>0,\\
\nabla\cdot u=0,\quad
x\in \Omega, t>0,\\
 \disp{\nabla n\cdot\nu=\nabla c\cdot\nu=0,u=0,}\quad
x\in \partial\Omega, t>0,\\
\disp{n(x,0)=n_0(x),c(x,0)=c_0(x),u(x,0)=u_0(x),}\quad
x\in \Omega,\\
 \end{array}\right.\label{1.1}
\end{equation}
where $\Omega\subseteq \mathbb{R}^3$ is a bounded    domain with smooth boundary.

In this paper, one of a
%
key role in our approach  is based on pursuing the time evolution of a coupled
functional of the form
$\int_{\Omega}n^{m-1}(\cdot,t) +\int_{\Omega} {c}^{2}(\cdot,t)+\int_{\Omega} |{u}(\cdot,t)|^{2}
 $
 (see Lemma \ref{lemmaghjssddgghhmk4563025xxhjklojjkkk})
which is a new (natural gradient-like energy functional) estimate of \dref{1.1}. 

This paper is organized as follows. In Section 2, we firstly give  the definition of weak solutions to \dref{1.1}, the  regularized problems of \dref{1.1} and
 state the main results of this paper and prove the local existence of classical solution to appropriately regularized problems of \dref{1.1}.
%
Section 3  and Section 4 will be devoted to an analysis of regularized problems of \dref{1.1}.
 On the basis of the compactness properties thereby implied, in Section 5 we shall finally
pass to the limit along an adequate sequence of numbers $\varepsilon = \varepsilon_j\searrow0$
and thereby verify the main results.

\section{Preliminaries and  main results}
Due to the strongly nonlinear term $(u \cdot \nabla)u$ and $\Delta n^m,$ the problem \dref{1.1} has no
classical solutions in general, and thus we consider its weak solutions in the following sense.
We first specify the notion of weak solution to which we will refer in the sequel.
\begin{definition}\label{df1}
Let $T > 0$ and $(n_0, c_0, u_0)$ fulfills
\dref{ccvvx1.731426677gg}.
Then a triple of functions $(n, c, u)$ is
called a weak solution of \dref{1.1} if the following conditions are satisfied
\begin{equation}
 \left\{\begin{array}{ll}
   n\in L_{loc}^1(\bar{\Omega}\times[0,T)),\\
    c \in L_{loc}^1([0,T); W^{1,1}(\Omega)),\\
u \in  L_{loc}^1([0,T); W^{1,1}(\Omega)),\\
 \end{array}\right.\label{dffff1.1fghyuisdakkklll}
\end{equation}
where $n\geq 0$ and $c\geq 0$ in
$\Omega\times(0, T)$ as well as $\nabla\cdot u = 0$ in the distributional sense in
 $\Omega\times(0, T)$,
moreover,
\begin{equation}\label{726291hh}
\begin{array}{rl}
 &u\otimes u \in L^1_{loc}(\bar{\Omega}\times [0, \infty);\mathbb{R}^{3\times 3})~~\mbox{and}~~~ n^m~\mbox{belong to}~~ L^1_{loc}(\bar{\Omega}\times [0, \infty)),\\
  &cu,~ ~nu ~~\mbox{and}~~n|\nabla c|~ \mbox{belong to}~~
L^1_{loc}(\bar{\Omega}\times [0, \infty);\mathbb{R}^{3})
\end{array}
\end{equation}
and
\begin{equation}
\begin{array}{rl}\label{eqx45xx12112ccgghh}
\disp{-\int_0^{T}\int_{\Omega}n\varphi_t-\int_{\Omega}n_0\varphi(\cdot,0)  }=&\disp{
\int_0^T\int_{\Omega}n^m\Delta\varphi+\int_0^T\int_{\Omega}n
\nabla c\cdot\nabla\varphi}\\
&+\disp{\int_0^T\int_{\Omega}nu\cdot\nabla\varphi}\\
\end{array}
\end{equation}
for any $\varphi\in C_0^{\infty} (\bar{\Omega}\times[0, T))$ satisfying
 $\frac{\partial\varphi}{\partial\nu}= 0$ on $\partial\Omega\times (0, T)$
  as well as
  \begin{equation}
\begin{array}{rl}\label{eqx45xx12112ccgghhjj}
\disp{-\int_0^{T}\int_{\Omega}c\varphi_t-\int_{\Omega}c_0\varphi(\cdot,0)  }=&\disp{-
\int_0^T\int_{\Omega}\nabla c\cdot\nabla\varphi-\int_0^T\int_{\Omega}c\varphi+\int_0^T\int_{\Omega}n\varphi+
\int_0^T\int_{\Omega}cu\cdot\nabla\varphi}\\
\end{array}
\end{equation}
for any $\varphi\in C_0^{\infty} (\bar{\Omega}\times[0, T))$  and
\begin{equation}
\begin{array}{rl}\label{eqx45xx12112ccgghhjjgghh}
\disp{-\int_0^{T}\int_{\Omega}u\varphi_t-\int_{\Omega}u_0\varphi(\cdot,0) -\kappa
\int_0^T\int_{\Omega} u\otimes u\cdot\nabla\varphi }=&\disp{-
\int_0^T\int_{\Omega}\nabla u\cdot\nabla\varphi-
\int_0^T\int_{\Omega}n\nabla\phi\cdot\varphi}\\
\end{array}
\end{equation}
for any $\varphi\in C_0^{\infty} (\bar{\Omega}\times[0, T);\mathbb{R}^3)$ fulfilling
$\nabla\varphi\equiv 0$ in
 $\Omega\times(0, T)$.
 If $\Omega\times (0,\infty)\longrightarrow \mathbb{R}^5$ is a weak solution of \dref{1.1} in
 $\Omega\times(0, T)$ for all $T > 0$, then we call
$(n, c, u)$ a global weak solution of \dref{1.1}.
\end{definition}

Throughout this paper,
we assume that 
\begin{equation}
\phi\in W^{1,\infty}(\Omega)
\label{dd1.1fghyuisdakkkllljjjkk}
\end{equation}
 and the initial data
$(n_0, c_0, u_0)$ fulfills
\begin{equation}\label{ccvvx1.731426677gg}
\left\{
\begin{array}{ll}
\displaystyle{n_0\in C^\kappa(\bar{\Omega})~~\mbox{for certain}~~ \kappa > 0~~ \mbox{with}~~ n_0\geq0 ~~\mbox{in}~~\Omega},\\
\displaystyle{c_0\in W^{1,\infty}(\Omega)~~\mbox{with}~~c_0\geq0~~\mbox{in}~~\bar{\Omega},}\\
\displaystyle{u_0\in D(A^\gamma_{r})~~\mbox{for~~ some}~~\gamma\in ( \frac{1}{2}, 1)~~\mbox{and any}~~ {r}\in (1,\infty),}\\
\end{array}
\right.
\end{equation}
where $A_{r}$ denotes the Stokes operator with domain $D(A_{r}) := W^{2,{r}}(\Omega)\cap  W^{1,{r}}_0(\Omega)
\cap L^{r}_{\sigma}(\Omega)$,
and
$L^{r}_{\sigma}(\Omega) := \{\varphi\in  L^{r}(\Omega)|\nabla\cdot\varphi = 0\}$ for ${r}\in(1,\infty)$
 (\cite{Sohr}).

\begin{theorem}\label{theorem3}
Let 
 \dref{dd1.1fghyuisdakkkllljjjkk}
 hold, and suppose that 
\begin{equation}\label{x1.73142vghf48}m> 2.
\end{equation}
Then for any choice of $n_0, c_0$ and $u_0$ fulfilling \dref{ccvvx1.731426677gg}, the problem \dref{1.1} possesses at least
one global weak solution $(n, c, u, P)$
 in the sense of Definition \ref{df1}. 
\end{theorem}
\begin{remark}
From Theorem \ref{theorem3}, we conclude that  if
the exponent $m$ of nonlinear diffusion  is large than $2$, then
 model \dref{1.1} exists a global solution, which implies the nonlinear diffusion term  benefits the global of solutions, which seems partly
 extends the results of Tao and Winkler \cite {Tao41215}, who proved the possibility of boundedness,
in the case that $m=1$, the coefficient of logistic source suitably large  and the strength of nonlinear fluid convection $\kappa=0$.

%
%
%
%
%
%
%
%
\end{remark}

Our intention is to construct a global weak solution of \dref{1.1} as the limit of smooth solutions of
appropriately regularized problems. To this end,
in order to deal with the strongly nonlinear term $(u \cdot \nabla)u$ and $\Delta n^m$,
%
we
need to introduce the following approximating equation of \dref{1.1}:
\begin{equation}
 \left\{\begin{array}{ll}
   n_{\varepsilon t}+u_{\varepsilon}\cdot\nabla n_{\varepsilon}=\Delta (n_{\varepsilon}+\varepsilon)^m-\nabla\cdot(n_{\varepsilon}\nabla c_{\varepsilon}),\quad
x\in \Omega, t>0,\\
    c_{\varepsilon t}+u_{\varepsilon}\cdot\nabla c_{\varepsilon}=\Delta c_{\varepsilon}-c_{\varepsilon}+n_{\varepsilon},\quad
x\in \Omega, t>0,\\
u_{\varepsilon t}+\nabla P_{\varepsilon}=\Delta u_{\varepsilon}-\kappa (Y_{\varepsilon}u_{\varepsilon} \cdot \nabla)u_{\varepsilon}+n_{\varepsilon}\nabla \phi,\quad
x\in \Omega, t>0,\\
\nabla\cdot u_{\varepsilon}=0,\quad
x\in \Omega, t>0,\\
 \disp{\nabla n_{\varepsilon}\cdot\nu=\nabla c_{\varepsilon}\cdot\nu=0,u_{\varepsilon}=0,\quad
x\in \partial\Omega, t>0,}\\
\disp{n_{\varepsilon}(x,0)=n_0(x),c_{\varepsilon}(x,0)=c_0(x),u_{\varepsilon}(x,0)=u_0(x)},\quad
x\in \Omega,\\
 \end{array}\right.\label{1.1fghyuisda}
\end{equation}
where
\begin{equation}
 \begin{array}{ll}
 Y_{\varepsilon}w := (1 + \varepsilon A)^{-1}w ~~~~\mbox{for all}~~ w\in L^2_{\sigma}(\Omega)
 \end{array}\label{aasddffgg1.1fghyuisda}
\end{equation}
is the standard Yosida approximation.
In light of the well-established fixed point arguments (see \cite{Winkler11215}, Lemma 2.1 of \cite{Painter55677} and Lemma 2.1 of \cite{Winkler51215}),
we can prove that \dref{1.1fghyuisda} is locally solvable in classical sense, which is stated as the following lemma.
%
%
\begin{lemma}\label{lemma70}
Assume
that $\varepsilon\in(0,1).$
%
Then there exist $T_{max,\varepsilon}\in  (0,\infty]$ and
a classical solution $(n_\varepsilon, c_\varepsilon, u_\varepsilon, P_\varepsilon)$ of \dref{1.1fghyuisda} in
$\Omega\times(0, T_{max,\varepsilon})$ such that
\begin{equation}
 \left\{\begin{array}{ll}
 n_\varepsilon\in C^0(\bar{\Omega}\times[0,T_{max,\varepsilon}))\cap C^{2,1}(\bar{\Omega}\times(0,T_{max,\varepsilon})),\\
  c_\varepsilon\in  C^0(\bar{\Omega}\times[0,T_{max,\varepsilon}))\cap C^{2,1}(\bar{\Omega}\times(0,T_{max,\varepsilon})),\\
  u_\varepsilon\in  C^0(\bar{\Omega}\times[0,T_{max,\varepsilon}))\cap C^{2,1}(\bar{\Omega}\times(0,T_{max,\varepsilon})),\\
  P_\varepsilon\in  C^{1,0}(\bar{\Omega}\times(0,T_{max,\varepsilon})),\\
   \end{array}\right.\label{1.1ddfghyuisda}
\end{equation}
 classically solving \dref{1.1fghyuisda} in $\Omega\times[0,T_{max,\varepsilon})$.
%
Moreover,  $n_\varepsilon$ and $c_\varepsilon$ are nonnegative in
$\Omega\times(0, T_{max,\varepsilon})$, and
\begin{equation}
\|n_\varepsilon(\cdot, t)\|_{L^\infty(\Omega)}+\|c_\varepsilon(\cdot, t)\|_{W^{1,\infty}(\Omega)}+\|A^\gamma u_\varepsilon(\cdot, t)\|_{L^{2}(\Omega)}\rightarrow\infty~~ \mbox{as}~~ t\rightarrow T_{max,\varepsilon},
\label{1.163072x}
\end{equation}
where $\gamma$ is given by \dref{ccvvx1.731426677gg}.
\end{lemma}

\section{A priori estimates}
In this section, we are going to establish an iteration step to develop the main ingredient of our result. The iteration depends on a series of a priori estimate.
The proof of this lemma is very similar to that of Lemmata 2.2 and 2.6  of  \cite{Tao41215}, so we omit its proof here.

\begin{lemma}\label{fvfgfflemma45}
There exists 
$\lambda > 0$ independent of $\varepsilon$ such that the solution of \dref{1.1fghyuisda} satisfies
%
%
\begin{equation}
\int_{\Omega}{n_{\varepsilon}}+\int_{\Omega}{c_{\varepsilon}}\leq \lambda~~\mbox{for all}~~ t\in(0, T_{max,\varepsilon}).
\label{ddfgczhhhh2.5ghju48cfg924ghyuji}
\end{equation}
%
%
\end{lemma}

\begin{lemma}\label{lemmaghjssddgghhmk4563025xxhjklojjkkk}
Let $m>2$.
Then there exists $C>0$ independent of $\varepsilon$ such that the solution of \dref{1.1fghyuisda} satisfies
\begin{equation}
\begin{array}{rl}
&\disp{\int_{\Omega}(n_{\varepsilon}+\varepsilon)^{m-1}+\int_{\Omega}   c_{\varepsilon}^2+\int_{\Omega}  | {u_{\varepsilon}}|^2\leq C~~~\mbox{for all}~~ t\in (0, T_{max,\varepsilon}).}\\
\end{array}
\label{czfvgb2.5ghhjuyuccvviihjj}
\end{equation}
In addition,
for each $T\in(0, T_{max,\varepsilon})$, one can find a constant $C > 0$ independent of $\varepsilon$ such that
\begin{equation}
\begin{array}{rl}
&\disp{\int_{0}^T\int_{\Omega} \left[ (n_{\varepsilon}+\varepsilon)^{2m-4} |\nabla {n_{\varepsilon}}|^2+ |\nabla {c_{\varepsilon}}|^2+ |\nabla {u_{\varepsilon}}|^2\right]\leq C.}\\
\end{array}
\label{bnmbncz2.5ghhjuyuivvbnnihjj}
\end{equation}
\end{lemma}
\begin{proof}
Taking ${c_{\varepsilon}}$ as the test function for the second  equation of \dref{1.1fghyuisda} and using $\nabla\cdot u_\varepsilon=0$ and the Young inequality  yields  that
\begin{equation}
\begin{array}{rl}
\disp\frac{1}{{2}}\disp\frac{d}{dt}\|{c_{\varepsilon}}\|^{{{2}}}_{L^{{2}}(\Omega)}+
\int_{\Omega} |\nabla c_{\varepsilon}|^2+ \int_{\Omega} | c_{\varepsilon}|^2=&\disp{\int_{\Omega} n_{\varepsilon}c_{\varepsilon}}\\
\leq&\disp{\frac{1}{2}\int_{\Omega} n_{\varepsilon}^2+\frac{1}{2}\int_{\Omega}c_{\varepsilon}^2.}\\
\end{array}
\label{hhxxcdfvvjjcz2.5}
\end{equation}
On the other hand, 
due to
 the Gagliardo--Nirenberg inequality, \dref{ddfgczhhhh2.5ghju48cfg924ghyuji}, in light of  the Young inequality and $m>2$, we obtain that for any $\delta_1>0,$
\begin{equation}
\begin{array}{rl}
\disp\|n_{\varepsilon}\|_{L^{2}(\Omega)}^2\leq&\disp{\|n_{\varepsilon}+\varepsilon\|_{L^{2}(\Omega)}^2}\\
=&\disp{\|(n_{\varepsilon}+\varepsilon)^{m-1}\|_{L^{\frac{2}{m-1}}(\Omega)}^{\frac{2}{m-1}}}\\
\leq&\disp{C_1\|\nabla (n_{\varepsilon}+\varepsilon)^{m-1}\|_{L^{2}(\Omega)}^{\frac{6}{6m-7}}\| (n_{\varepsilon}+\varepsilon)^{m-1}\|_{L^{\frac{1}{m-1}}(\Omega)}^{\frac{2}{m-1}-\frac{6}{6m-7}}}\\
\leq&\disp{C_2(\|\nabla (n_{\varepsilon}+\varepsilon)^{m-1}\|_{L^{2}(\Omega)}^{\frac{6}{6m-7}}+1)}\\
\leq&\disp{\delta_1\|\nabla (n_{\varepsilon}+\varepsilon)^{m-1}\|_{L^{2}(\Omega)}^{2}+C_3~~\mbox{for all}~~ t\in(0, T_{max,\varepsilon})}\\
\end{array}
\label{5566ddddfgcz2.5ghju4cddfff8cfg924gjjkkkhyuji}
\end{equation}
with some positive constants  $C_1,C_2$ and $C_3$ independent of $\varepsilon$.
Hence, in light of \dref{hhxxcdfvvjjcz2.5} and \dref{5566ddddfgcz2.5ghju4cddfff8cfg924gjjkkkhyuji}, we derive that
\begin{equation}
\begin{array}{rl}
\disp\disp\frac{d}{dt}\|{c_{\varepsilon}}\|^{{{2}}}_{L^{{2}}(\Omega)}+2
\int_{\Omega} |\nabla c_{\varepsilon}|^2+ \int_{\Omega}  c_{\varepsilon}^2\leq&\disp{\delta_1\|\nabla (n_{\varepsilon}+\varepsilon)^{m-1}\|_{L^{2}(\Omega)}^{2}+C_3~~\mbox{for all}~~ t\in(0, T_{max,\varepsilon})}\\
\end{array}
\label{hhxxcdfvv55677dd55678fghhjjcz2.5}
\end{equation}
and some positive constant $C_{3}$ independent of $\varepsilon.$
Next, multiply the first equation in $\dref{1.1fghyuisda}$ by $({n_{\varepsilon}}+\varepsilon)^{m-2}$
 and combining with the second equation, using $\nabla\cdot u_\varepsilon=0$ and the Young inequality implies that
\begin{equation}
\begin{array}{rl}
&\disp{\frac{1}{{m-1}}\frac{d}{dt}\|{n_{\varepsilon}}+\varepsilon\|^{{{m-1}}}_{L^{{m-1}}(\Omega)}+
m(m-2)\int_{\Omega}({n_{\varepsilon}}+\varepsilon)^{2m-4} |\nabla n_{\varepsilon}|^2}\\
\leq&\disp{(m-2)\int_\Omega ({n_{\varepsilon}}+\varepsilon)^{m-2}|\nabla n_{\varepsilon}||\nabla c_{\varepsilon}| }\\
 \leq&\disp{\frac{m(m-2)}{2}\int_{\Omega}({n_{\varepsilon}}+\varepsilon)^{2m-4} |\nabla n_{\varepsilon}|^2+\frac{(m-2)}{2m}\int_\Omega |\nabla c_{\varepsilon}|^2 .}
\end{array}
\label{55hhjjcz2.5}
\end{equation}
Now, multiplying the
third equation of \dref{1.1fghyuisda} by $u_\varepsilon$, integrating by parts and using $\nabla\cdot u_{\varepsilon}=0$, we derive that
\begin{equation}
\frac{1}{2}\frac{d}{dt}\int_{\Omega}{|u_{\varepsilon}|^2}+\int_{\Omega}{|\nabla u_{\varepsilon}|^2}= \int_{\Omega}n_{\varepsilon}u_{\varepsilon}\cdot\nabla \phi~~\mbox{for all}~~ t\in(0, T_{max,\varepsilon}).
\label{55ddddfgcz2.5ghju48cfg924ghyuji}
\end{equation}
Here we use the H\"{o}lder inequality and \dref{dd1.1fghyuisdakkkllljjjkk}
and the continuity of the embedding $W^{1,2}(\Omega)\hookrightarrow L^6(\Omega)$ and
 to
find $C_4>0$ and $C_5 > 0$ such that
\begin{equation}
\begin{array}{rl}
\disp\int_{\Omega}n_{\varepsilon}u_{\varepsilon}\cdot\nabla \phi\leq&\disp{\|\nabla \phi\|_{L^\infty(\Omega)}\|n_{\varepsilon}\|_{L^{\frac{6}{5}}(\Omega)}\|\nabla u_{\varepsilon}\|_{L^{2}(\Omega)}}\\
\leq&\disp{C_4\|n_{\varepsilon}\|_{L^{\frac{6}{5}}(\Omega)}\| \nabla u_{\varepsilon}\|_{L^{2}(\Omega)}}\\
\leq&\disp{\frac{C_4^2}{2}\|n_{\varepsilon}\|_{L^{\frac{6}{5}}(\Omega)}^2+\frac{1}{2}\|\nabla u_{\varepsilon}\|_{L^{2}(\Omega)}^2}\\
\leq&\disp{C_5\|n_{\varepsilon}\|_{L^{2}(\Omega)}^2+\frac{1}{2}\|\nabla u_{\varepsilon}\|_{L^{2}(\Omega)}^2~~\mbox{for all}~~ t\in(0, T_{max,\varepsilon}),}\\
\end{array}
\label{55ddddfgcz2.5ghju48cfg924ghyuji}
\end{equation}
which together with \dref{5566ddddfgcz2.5ghju4cddfff8cfg924gjjkkkhyuji} implies that for any $\delta_2$
\begin{equation}
\begin{array}{rl}
\disp\int_{\Omega}n_{\varepsilon}u_{\varepsilon}\cdot\nabla \phi\leq&\disp{\frac{\delta_2}{2}\|\nabla (n_{\varepsilon}+\varepsilon)^{m-1}\|_{L^{2}(\Omega)}^{2}+\frac{1}{2}\|\nabla u_{\varepsilon}\|_{L^{2}(\Omega)}^2+C_6~~\mbox{for all}~~ t\in(0, T_{max,\varepsilon}),}\\
\end{array}
\label{2234555ddddfgcz2.5ghju48cfg924ghyuji}
\end{equation}
where $C_6$ is a positive constant independent of $\varepsilon.$
Inserting \dref{2234555ddddfgcz2.5ghju48cfg924ghyuji} into \dref{55ddddfgcz2.5ghju48cfg924ghyuji} and using the Young inequality and $m>2$, we conclude that there exists a positive constant $C_7$  such that
\begin{equation}
\frac{d}{dt}\int_{\Omega}{|u_{\varepsilon}|^2}+\int_{\Omega}{|\nabla u_{\varepsilon}|^2}\leq \delta_2\|\nabla (n_{\varepsilon}+\varepsilon)^{m-1}\|_{L^{2}(\Omega)}^{2}+C_7~~\mbox{for all}~~ t\in(0, T_{max,\varepsilon}).\\
\label{44555ddddf334799gcz2.5ghju48cfg924ghyuji}
\end{equation}
Take an evident linear combination of the inequalities provided by \dref{hhxxcdfvv55677dd55678fghhjjcz2.5}, \dref{55hhjjcz2.5} and \dref{44555ddddf334799gcz2.5ghju48cfg924ghyuji}, we conclude
\begin{equation}
 \begin{array}{rl}
 &\disp{\frac{d}{dt}\left(\|{c_{\varepsilon}}\|^{{{2}}}_{L^{{2}}(\Omega)}+
 \frac{L}{{m-1}}\|{n_{\varepsilon}}+\varepsilon\|^{{{m-1}}}_{L^{{m-1}}(\Omega)}+\int_{\Omega}|u_{\varepsilon}|^2\right)+\int_{\Omega}|\nabla u_{\varepsilon}|^2}\\
 &\disp{+(L\frac{m(m-2)}{(m-1)^2}-\delta_2-\delta_1)\|\nabla (n_{\varepsilon}+\varepsilon)^{m-1}\|_{L^{2}(\Omega)}^{2}
 +(2-L\frac{(m-2)}{2m})
\int_{\Omega} |\nabla c_{\varepsilon}|^2+ \int_{\Omega} c_{\varepsilon}^2}\\
 \leq&\disp{ C_8~~~\mbox{for all}~~t\in (0, T_{max,\varepsilon}),}\\
\end{array}\label{vgbc45678cvbbffeerfgghhjuuloollgghhhyhh}
\end{equation}
where $C_8$ and $L$ are positive constants.
Now, choosing $L=\frac{2m}{(m-2)}$ and $\delta_1=\delta_2=\frac{L}{4}\frac{m(m-2)}{(m-1)^2}$ 
 in \dref{vgbc45678cvbbffeerfgghhjuuloollgghhhyhh}, we can conclude that
\dref{czfvgb2.5ghhjuyuccvviihjj}
and \dref{bnmbncz2.5ghhjuyuivvbnnihjj}.
\end{proof}
With the help of Lemma \ref{lemmaghjssddgghhmk4563025xxhjklojjkkk}, in light of the Gagliardo--Nirenberg inequality and an application of well-known arguments
from parabolic regularity theory, we can derive the following Lemma:
\begin{lemma}\label{lemmddaghjssddgghhmk4563025xxhjklojjkkk}
Let $m>2$.
Then there exists $C>0$ independent of $\varepsilon$ such that the solution of \dref{1.1fghyuisda} satisfies
\begin{equation}
\begin{array}{rl}
&\disp{\int_{\Omega}   c_{\varepsilon}^{\frac{8(m-1)}{3}}\leq C~~~\mbox{for all}~~ t\in (0, T_{max,\varepsilon}).}\\
\end{array}
\label{czfvgb2.5ghffghjuyuccvviihjj}
\end{equation}
In addition,
for each $T\in(0, T_{max,\varepsilon})$, one can find a constant $C > 0$ independent of $\varepsilon$ such that
\begin{equation}
\begin{array}{rl}
&\disp{\int_{0}^T\int_{\Omega} \left[n_{\varepsilon}^{\frac{8(m-1)}{3}}+{c_{\varepsilon}^{\frac{8m-14}{3}}}|\nabla c_{\varepsilon}|^2+ c_{\varepsilon}^{\frac{40(m-1)}{9}}\right]\leq C.}\\
\end{array}
\label{bnmbncz2.ffghh5ghhjuyuivvbnnihjj}
\end{equation}
\end{lemma}
\begin{proof}
Firstly, due to \dref{czfvgb2.5ghhjuyuccvviihjj}
and \dref{bnmbncz2.5ghhjuyuivvbnnihjj}, in light of the Gagliardo--Nirenberg inequality, for some $C_1$ and $C_2> 0$ which are independent of $\varepsilon$, we derive that
\begin{equation}
\begin{array}{rl}
\disp\int_{0}^T\disp\int_{\Omega}(n_{\varepsilon}+\varepsilon)^{\frac{8(m-1)}{3}} =&\disp{\int_{0}^T\| {(n_{\varepsilon}+\varepsilon)^{m-1}}\|^{{\frac{8}{3}}}_{L^{\frac{8}{3}}(\Omega)}}\\
\leq&\disp{C_{1}\int_{0}^T\left(\| \nabla{(n_{\varepsilon}+\varepsilon)^{m-1}}\|^{2}_{L^{2}(\Omega)}\|{(n_{\varepsilon}+\varepsilon)^{m-1}}\|^{{\frac{2}{3}}}_{L^{1}(\Omega)}+
\|{(n_{\varepsilon}+\varepsilon)^{m-1}}\|^{{\frac{8}{3}}}_{L^{1}(\Omega)}\right)}\\
\leq&\disp{C_{2}(T+1)~~\mbox{for all}~~ T > 0.}\\
\end{array}
\label{ddffbnmbnddfgcz2ddfvgbhh.htt678ddfghhhyuiihjj}
\end{equation}
Next, taking  ${c_{\varepsilon}^{\frac{8m-11}{3}}}$ as the test function for the second  equation of \dref{1.1fghyuisda} and using $\nabla\cdot u_\varepsilon=0$ and the Young inequality  yields  that
\begin{equation}
\begin{array}{rl}
&\disp\frac{3}{8(m-1)}\disp\frac{d}{dt}\|{c_{\varepsilon}}\|^{{{\frac{8(m-1)}{3}}}}_{L^{{\frac{8(m-1)}{3}}}(\Omega)}+\frac{8m-11}{3}
\int_{\Omega} {c_{\varepsilon}^{\frac{8m-14}{3}}}|\nabla c_{\varepsilon}|^2+ \int_{\Omega} c_{\varepsilon}^{\frac{8(m-1)}{3}}\\
=&\disp{\int_{\Omega} n_{\varepsilon}c_{\varepsilon}^{\frac{8m-11}{3}}}\\
\leq&\disp{C_{3}\int_{\Omega} n_{\varepsilon}^{\frac{8(m-1)}{3}}+\frac{1}{2}\int_{\Omega}c_{\varepsilon}^{\frac{8(m-1)}{3}}~~~\mbox{for all}~~t\in (0, T_{max,\varepsilon})}\\
\end{array}
\label{hhxxcdfssxxdccffgghvvjjcz2.5}
\end{equation}
with  some positive constant $C_{3}.$
Hence, due to \dref{ddffbnmbnddfgcz2ddfvgbhh.htt678ddfghhhyuiihjj} and \dref{hhxxcdfssxxdccffgghvvjjcz2.5}, we can find $C_{4} > 0$ such that
\begin{equation}
\begin{array}{rl}
&\disp{\int_{\Omega}   c_{\varepsilon}^{{{\frac{8(m-1)}{3}}}}\leq C_{4}~~~\mbox{for all}~~ t\in (0, T_{max,\varepsilon})}\\
\end{array}
\label{czfvgb2.5ghhffghhjjjuyuccvviihjj}
\end{equation}
and
\begin{equation}
\begin{array}{rl}
&\disp{\int_{0}^T\int_{\Omega}  {c_{\varepsilon}^{\frac{8m-14}{3}}}|\nabla c_{\varepsilon}|^2\leq C_{4}(T+1)~~~\mbox{for all}~~T\in(0, T_{max,\varepsilon}).}\\
\end{array}
\label{bnmbnczddffhhjj2.5ghhjuyuivvbnnihjj}
\end{equation}
Now, due to \dref{czfvgb2.5ghhffghhjjjuyuccvviihjj}
and \dref{bnmbnczddffhhjj2.5ghhjuyuivvbnnihjj}, in light of the Gagliardo--Nirenberg inequality, we derive that there exist positive constants $C_{5}$ and $C_{6}$ such that
\begin{equation}
\begin{array}{rl}
\disp\int_{0}^T\disp\int_{\Omega} c_{\varepsilon}^{\frac{40(m-1)}{9}} =&\disp{\int_{0}^T\| { c_{\varepsilon}^{\frac{4(m-1)}{3}}}\|^{{\frac{10}{3}}}_{L^{\frac{10}{3}}(\Omega)}}\\
\leq&\disp{C_{5}\int_{0}^T\left(\| \nabla{ c_{\varepsilon}^{\frac{4(m-1)}{3}}}\|^{2}_{L^{2}(\Omega)}\|{ c_{\varepsilon}^{\frac{4(m-1)}{3}}}\|^{{\frac{4}{3}}}_{L^{2}(\Omega)}+
\|{ c_{\varepsilon}^{\frac{4(m-1)}{3}}}\|^{{\frac{10}{3}}}_{L^{2}(\Omega)}\right)}\\
\leq&\disp{C_{6}(T+1)~~\mbox{for all}~~ T > 0.}\\
\end{array}
\label{ddffbnmbnddfgffggjjkkuuiicz2ddfvgbhh.htt678ddfghhhyuiihjj}
\end{equation}
Finally, collecting \dref{ddffbnmbnddfgcz2ddfvgbhh.htt678ddfghhhyuiihjj},
and \dref{czfvgb2.5ghhffghhjjjuyuccvviihjj}--\dref{ddffbnmbnddfgffggjjkkuuiicz2ddfvgbhh.htt678ddfghhhyuiihjj}, we can get the results.
\end{proof}

\begin{lemma}\label{lemma630jklhhjj}
There  exists a positive constant $C:=C(\varepsilon)$ depends on $\varepsilon$ such that
\begin{equation}
\int_{\Omega}{|\nabla u_{\varepsilon}(\cdot,t)|^2}\leq C~~\mbox{for all}~~ t\in(0, T_{max,\varepsilon})
\label{ddxcvbbggddfgcz2vv.5ghju48cfg924ghyuji}
\end{equation}
and
\begin{equation}
\int_0^{T}\int_{\Omega}{|\Delta u_{\varepsilon}|^2}\leq C~~\mbox{for all}~~ T\in(0, T_{max,\varepsilon}).
\label{gghhddddcz2.5ghju48cfg924}
\end{equation}
\end{lemma}
\begin{proof}
Firstly, due to  $D(1 + \varepsilon A)  :=W^{2,2}(\Omega) \cap W_{0,\sigma}^{1,2}(\Omega)\hookrightarrow L^\infty(\Omega),$ by \dref{czfvgb2.5ghhjuyuccvviihjj}, we derive that for some  $C_1> 0$ and $C_2 > 0$,
\begin{equation}
\|Y_{\varepsilon}u_{\varepsilon}\|_{L^\infty(\Omega)}=\|(I+\varepsilon A)^{-1}u_{\varepsilon}\|_{L^\infty(\Omega)}\leq C_1\|u_{\varepsilon}(\cdot,t)\|_{L^2(\Omega)}\leq C_2~~\mbox{for all}~~t\in(0,T_{max,\varepsilon}).
\label{ssdcfvgdhhjjdfghgghjjnnhhkklld911cz2.5ghju48}
\end{equation}
Next, let $A = -\mathscr{2}\Delta$; testing the third equation by $Au_{\varepsilon}$ implies
\begin{equation}
\begin{array}{rl}
&\disp{\frac{1}{{2}}\frac{d}{dt}\|A^{\frac{1}{2}}u_{\varepsilon}\|^{{{2}}}_{L^{{2}}(\Omega)}+
\int_{\Omega}|Au_{\varepsilon}|^2 }\\
=&\disp{ \int_{\Omega}Au_{\varepsilon}\kappa
(Y_{\varepsilon}u_{\varepsilon} \cdot \nabla)u_{\varepsilon}+ \int_{\Omega}n_{\varepsilon}\nabla\phi Au_{\varepsilon}}\\
\leq&\disp{ \frac{1}{2}\int_{\Omega}|Au_{\varepsilon}|^2+\kappa^2\int_{\Omega}
|(Y_{\varepsilon}u_{\varepsilon} \cdot \nabla)u_{\varepsilon}|^2+ \|\nabla\phi\|^2_{L^\infty(\Omega)}\int_{\Omega}n_{\varepsilon}^2~~\mbox{for all}~~t\in(0,T_{max,\varepsilon}).}\\
\end{array}
\label{ddfghgghjjnnhhkklld911cz2.5ghju48}
\end{equation}

On the other hand, in light of the Gagliardo--Nirenberg inequality, the Young inequality and \dref{ssdcfvgdhhjjdfghgghjjnnhhkklld911cz2.5ghju48}, there exists a positive constant $C_3$
such that
\begin{equation}
\begin{array}{rl}
\kappa^2\disp\int_{\Omega}
|(Y_{\varepsilon}u_{\varepsilon} \cdot \nabla)u_{\varepsilon}|^2\leq&\disp{ \kappa^2\|Y_{\varepsilon}u_{\varepsilon}\|^2_{L^\infty(\Omega)}\int_{\Omega}|\nabla u_{\varepsilon}|^2}\\
\leq&\disp{ \kappa^2\|Y_{\varepsilon}u_{\varepsilon}\|^2_{L^\infty(\Omega)}\int_{\Omega}|\nabla u_{\varepsilon}|^2}\\
\leq&\disp{ C_3\int_{\Omega}|\nabla u_{\varepsilon}|^2~~\mbox{for all}~~t\in(0,T_{max,\varepsilon}).}\\
\end{array}
\label{ssdcfvgddfghgghjjnnhhkklld911cz2.5ghju48}
\end{equation}
Here we have the well-known fact that $\|A(\cdot)\|_{L^{2}(\Omega)}$ defines a norm
equivalent to $\|\cdot\|_{W^{2,2}(\Omega)}$ on $D(A)$ (see Theorem 2.1.1 of \cite{Sohr}).
Therefore, recalling that $A = -\mathscr{2}\Delta$ and hence
$$\|A^{\frac{1}{2}}u_{\varepsilon}\|^{{{2}}}_{L^{{2}}(\Omega)} = \|\nabla u_{\varepsilon}\|^{{{2}}}_{L^{{2}}(\Omega)},$$
inserting the above equation  and \dref{ssdcfvgddfghgghjjnnhhkklld911cz2.5ghju48} into \dref{ddfghgghjjnnhhkklld911cz2.5ghju48}, we can conclude that
\begin{equation}
\begin{array}{rl}
&\disp{\frac{1}{{2}}\frac{d}{dt}\|\nabla u_{\varepsilon}\|^{{{2}}}_{L^{{2}}(\Omega)}+
\int_{\Omega}|\Delta u_{\varepsilon}|^2 \leq C_4\int_{\Omega}|\nabla u_{\varepsilon}|^2+ \|\nabla\phi\|^2_{L^\infty(\Omega)}\int_{\Omega}
n_{\varepsilon}^2~~\mbox{for all}~~t\in(0,T_{max,\varepsilon})}\\
\end{array}
\label{ddfgghhddfghgghjjnnhhkklld911cz2.5ghju48}
\end{equation}
with some positive constant $C_4.$
Collecting \dref{ddffbnmbnddfgcz2ddfvgbhh.htt678ddfghhhyuiihjj} and \dref{ddfgghhddfghgghjjnnhhkklld911cz2.5ghju48} and applying the Young inequality, we can get
 the results.
\end{proof}

\begin{lemma}\label{xccffgghhlemma4563025xxhjkloghyui}
There exists $C:=C(\varepsilon)> 0$ depends on $\varepsilon$ such that
\begin{equation}
\int_{\Omega}{|\nabla c_{\varepsilon}(\cdot,t)|^2}\leq C~~\mbox{for all}~~ t\in(0, T_{max,\varepsilon})
\label{ddxxxcvvddcvddffbbggddfgcz2vv.5ghju48cfg924ghyuji}
\end{equation}
and
\begin{equation}
\int_0^{T}\int_{\Omega}{|\Delta c_{\varepsilon}|^2}\leq C~~\mbox{for all}~~ T\in(0, T_{max,\varepsilon}).
\label{gddffghhddddcz2.gghh5ghju48cfg924}
\end{equation}
\end{lemma}
\begin{proof}
Firstly,  testing the second equation in \dref{1.1fghyuisda} against $-\Delta c_{\varepsilon}$  and employing the Young inequality yields
\begin{equation}
\begin{array}{rl}
\disp{\frac{1}{{2}}\frac{d}{dt} \|\nabla c_{\varepsilon}\|^{{{2}}}_{L^{{2}}(\Omega)}}= &\disp{\int_{\Omega}  -\Delta c_{\varepsilon}(\Delta c_{\varepsilon}-c_{\varepsilon}+n_{\varepsilon}-u_{\varepsilon}\cdot\nabla  c_{\varepsilon})}
\\
=&\disp{-\int_{\Omega}  |\Delta c_{\varepsilon}|^2-\int_{\Omega} |\nabla c_{\varepsilon}|^{2}-\int_\Omega n_{\varepsilon}\Delta c_{\varepsilon}-\int_\Omega (u_{\varepsilon}\cdot\nabla  c_{\varepsilon})\Delta c_{\varepsilon}}\\
\leq&\disp{-\frac{1}{2}\int_{\Omega}  |\Delta c_{\varepsilon}|^2-\int_{\Omega} |\nabla c_{\varepsilon}|^{2}+\int_\Omega n_{\varepsilon}^2+\int_\Omega |u_{\varepsilon}|^2|\nabla  c_{\varepsilon}|^2}\\
\end{array}
\label{cz2.5ghju48156}
\end{equation}
for all $t\in(0,T_{max,\varepsilon})$.
Now, employing  \dref{czfvgb2.5ghhjuyuccvviihjj} and \dref{gghhddddcz2.5ghju48cfg924},  the Gagliardo--Nirenberg inequality and the Young inequality, we derive there exist positive constants
$C_1,C_2$ and $C_3$ such that
 \begin{equation}
\begin{array}{rl}
\disp{\int_\Omega |u_{\varepsilon}|^2|\nabla  c_{\varepsilon}|^2}
=&\disp{\|u_{\varepsilon}\|^{2}_{L^{8}(\Omega)}\|\nabla c_{\varepsilon}\|^{2}_{L^{\frac{8}{3}}(\Omega)}}\\
\leq&\disp{\|u_{\varepsilon}\|^{2}_{L^{8}(\Omega)}C_1(\|\Delta c_{\varepsilon}\|^{\frac{11}{8}}_{L^2(\Omega)}\| c_{\varepsilon}\|^{\frac{5}{8}}_{L^2(\Omega)}+\| c_{\varepsilon}\|^{2}_{L^2(\Omega)})}\\
\leq&\disp{\|u_{\varepsilon}\|^{2}_{L^{8}(\Omega)}C_2(\|\Delta c_{\varepsilon}\|^{\frac{11}{8}}_{L^2(\Omega)}+1)}\\
\leq&\disp{\frac{1}{4}\|\Delta c_{\varepsilon}\|^{2}_{L^2(\Omega)}+C_3(\|u_{\varepsilon}\|^{\frac{32}{5}}_{L^{8}(\Omega)}+1)}\\
\end{array}
\label{dd11cfvggcz2.5ghju48156}
\end{equation}
for all $t\in(0, T_{max,\varepsilon})$. 
Now, in view of the Gagliardo-Nirenberg inequality and the well-known fact that
$\|A(\cdot)\|_{L^2(\Omega)}$ defines a norm equivalent to $\|\cdot\|_{W^{2,2}(\Omega)}$
 on $W^{2,2}(\Omega)\cap W_0^{1,2}(\Omega)$ (see p. 129, Theorem e of \cite{Sohr}), we have
\begin{equation}
\begin{array}{rl}
\disp{C_3\|u_{\varepsilon}\|^{\frac{32}{5}}_{L^{8}(\Omega)}}
\leq&\disp{C_3\|A u_{\varepsilon}\|^{\frac{4}{5}}_{L^{2}(\Omega)}\| u_{\varepsilon}\|^{\frac{28}{5}}_{L^{6}(\Omega)}}\\
\leq&\disp{C_4(\|A u_{\varepsilon}\|^{2}_{L^{2}(\Omega)}+1),}\\
\end{array}
\label{dccvbbnnd11cfvggcz2.5ghju48156}
\end{equation}
where $C_4$ is a  positive constant.
Hence, in together with \dref{dccvbbnnd11cfvggcz2.5ghju48156} and \dref{gghhddddcz2.5ghju48cfg924}, we conclude there exists a positive constant $C_5$ such that
for all $T\in(0, T_{max,\varepsilon})$,
\begin{equation}
C_3\int_0^{T}{\|u_{\varepsilon}\|^{\frac{32}{5}}_{L^{8}(\Omega)}}\leq C_5.
\label{gddfggbnmmfghhddddcz2.5ghju48cfg924}
\end{equation}
Inserting \dref{dccvbbnnd11cfvggcz2.5ghju48156} and \dref{dd11cfvggcz2.5ghju48156}     into \dref{cz2.5ghju48156} and  using \dref{ddffbnmbnddfgcz2ddfvgbhh.htt678ddfghhhyuiihjj} and \dref{gddfggbnmmfghhddddcz2.5ghju48cfg924}, we can derive
\dref{ddxxxcvvddcvddffbbggddfgcz2vv.5ghju48cfg924ghyuji}
and \dref{gddffghhddddcz2.gghh5ghju48cfg924}.
This completes the proof of Lemma \ref{xccffgghhlemma4563025xxhjkloghyui}.
\end{proof}

With Lemmata \ref{lemmaghjssddgghhmk4563025xxhjklojjkkk}--\ref{xccffgghhlemma4563025xxhjkloghyui} at hand, we are now in the position to prove  the solution of approximate
problem \dref{1.1fghyuisda}  is actually global in time.
\begin{lemma}\label{kkklemmaghjmk4563025xxhjklojjkkk}
Let $m>2$. Then
for all $\varepsilon\in(0,1),$ the solution of  \dref{1.1fghyuisda} is global in time.
\end{lemma}
\begin{proof}
Assuming that $T_{max,\varepsilon}$ be finite for some $\varepsilon\in(0,1)$.
Next,
applying almost exactly the same arguments as in the proof of Lemma 3.4   in \cite{Zhengsdsd6},
 we may derive the following  estimate:
 the solution of \dref{1.1fghyuisda} satisfies that for all $\beta>1$
 \begin{equation}
\begin{array}{rl}
&\disp{\frac{1}{2\beta}\frac{d}{dt}\int_{\Omega} |\nabla c_{\varepsilon}|^{2\beta} +
\int_{\Omega} |\nabla c_{\varepsilon}|^{2\beta}+\frac{1}{2}\int_\Omega  |\nabla c_{\varepsilon}|^{2\beta-2}|D^2c_{\varepsilon}|^2+
\frac{(\beta-1)}{2\beta^2}\|\nabla |\nabla c_{\varepsilon}|^\beta\|_{L^2(\Omega)}^2}\\
\leq&\disp{C_1\int_\Omega n^2_{\varepsilon} |\nabla c_{\varepsilon}|^{2\beta-2}
+\int_\Omega |Du_{\varepsilon}| |\nabla c_{\varepsilon}|^{2\beta}+C_1~~\mbox{for all}~~ t\in(0, T_{max,\varepsilon}),}\\
\end{array}
\label{cz2.5715ssdergghhffghhhtyuftgg1hhkkhhggjjllll}
\end{equation}
 where $C_1$ is a positive constant.
On the other hand, due to \dref{ddxcvbbggddfgcz2vv.5ghju48cfg924ghyuji}, we derive that there exists a positive constant $C_2$ such that
\begin{equation}
\|D u_{\varepsilon}(\cdot,t)\|_{L^2(\Omega)}\leq C_2~~\mbox{for all}~~ t\in(0, T_{max,\varepsilon}).
\label{ddxcvbbggdddfghhdfgcz2vv.5ghju48cfg924ghyuji}
\end{equation}
Hence , in light of  the  H\"{o}lder inequality and the Gagliardo--Nirenberg inequality, \dref{ddxxxcvvddcvddffbbggddfgcz2vv.5ghju48cfg924ghyuji} and the Young inequality, we conclude that
\begin{equation}
\begin{array}{rl}
\disp\int_\Omega |Du_{\varepsilon}| |\nabla c_{\varepsilon}|^{2\beta}\leq &\disp{C_2\|\nabla c_{\varepsilon}\|^{2\beta}_{L^{4\beta}(\Omega)}}\\
= &\disp{C_2\||\nabla c_{\varepsilon}|^\beta\|^{2}_{L^{4}(\Omega)}}\\
= &\disp{C_2(\|\nabla |\nabla c_{\varepsilon}|^\beta\|_{L^2(\Omega)}^{\frac{6\beta-3}{6\beta-2}}\| |\nabla c_{\varepsilon}|^\beta\|_{L^\frac{2}{\beta}(\Omega)}^{\frac{6\beta-1}{6\beta-2}}+\|   |\nabla c_{\varepsilon}|^\beta\|_{L^\frac{2}{\beta}
(\Omega)}^{2})}\\
\leq &\disp{C_3(\|\nabla |\nabla c_{\varepsilon}|^\beta\|_{L^2(\Omega)}^{\frac{6\beta-3}{6\beta-2}}+1)}\\
\leq &\disp{\frac{(\beta-1)}{8\beta^2}\|\nabla |\nabla c_{\varepsilon}|^\beta\|_{L^2(\Omega)}^{2}+C_4~~\mbox{for all}~~ t\in(0, T_{max,\varepsilon})}\\
\end{array}
\label{cz2.5715ssdeddertrftgg1hhkkhhggjjllll}
\end{equation}
with some positive constants $C_3$
and $C_4.$
Now, inserting \dref{cz2.5715ssdeddertrftgg1hhkkhhggjjllll} into  \dref{cz2.5715ssdergghhffghhhtyuftgg1hhkkhhggjjllll}, we derive that there exists a positive constant $C_5$ such that
\begin{equation}
\begin{array}{rl}
&\disp{\frac{1}{2\beta}\frac{d}{dt}\int_{\Omega} |\nabla c_{\varepsilon}|^{2\beta} +
\int_{\Omega} |\nabla c_{\varepsilon}|^{2\beta}+\frac{1}{2}\int_\Omega  |\nabla c_{\varepsilon}|^{2\beta-2}|D^2c_{\varepsilon}|^2+
\frac{3(\beta-1)}{8\beta^2}\|\nabla |\nabla c_{\varepsilon}|^\beta\|_{L^2(\Omega)}^2}\\
\leq&\disp{C_1\int_\Omega n^2_{\varepsilon} |\nabla c_{\varepsilon}|^{2\beta-2}
+C_5~~\mbox{for all}~~ t\in(0, T_{max,\varepsilon}).}\\
\end{array}
\label{111cz2.5715ssdergghhffghhhtyuftgg1hhkkhhggjjllll}
\end{equation}
Next, with the help of the Young inequality, we derive that there exists a positive constant $C_6$ such that
\begin{equation}
\begin{array}{rl}
\disp{C_1\int_\Omega n^2_{\varepsilon} |\nabla c_{\varepsilon}|^{2\beta-2}\leq
\frac{1}{4}\int_\Omega  |\nabla c_{\varepsilon}|^{4(2\beta-2)}+C_6\int_\Omega n^{\frac{8}{3}}_{\varepsilon}+C_1.}\\
\end{array}
\label{111cz2.571ffg5ssdergghhffg334566hhhtyuftgg1hhkkhhggjjllll}
\end{equation}
Now, choosing $\beta=\frac{4}{3}$, in \dref{111cz2.5715ssdergghhffghhhtyuftgg1hhkkhhggjjllll} and \dref{111cz2.571ffg5ssdergghhffg334566hhhtyuftgg1hhkkhhggjjllll},
 we conclude that
\begin{equation}
\begin{array}{rl}
&\disp{\frac{1}{\frac{8}{3}}\frac{d}{dt}\int_{\Omega} |\nabla c_{\varepsilon}|^{\frac{8}{3}} +\frac{3}{4}
\int_{\Omega} |\nabla c_{\varepsilon}|^{\frac{8}{3}}+\frac{1}{2}\int_\Omega  |\nabla c_{\varepsilon}|^{\frac{2}{3}}|D^2c_{\varepsilon}|^2+
\frac{3}{16}\|\nabla |\nabla c_{\varepsilon}|^\frac{4}{3}\|_{L^2(\Omega)}^2}\\
\leq&\disp{C_6\int_\Omega n^{\frac{8}{3}}_{\varepsilon}+C_1.}\\
\end{array}
\label{cz2.5715ssdergssdrffttghhffghhhtyuftgg1hhkkhhggjjllll}
\end{equation}
Here we have use the fact that $4(2\beta-2)=2\beta.$
Hence, in light of \dref{ddffbnmbnddfgcz2ddfvgbhh.htt678ddfghhhyuiihjj} and $m>2$, by \dref{cz2.5715ssdergssdrffttghhffghhhtyuftgg1hhkkhhggjjllll}, we derive that there exists a positive constant $C_7$ such that
\begin{equation}
\begin{array}{rl}
&\disp{\|\nabla c_\varepsilon(\cdot, t)\|_{L^{\frac{8}{3}}(\Omega)}\leq C_7~~ \mbox{for all}~~ t\in(0,T_{max,\varepsilon})}\\
\end{array}
\label{zjccffgbhhjvcc22344789c4456788cvvvbbvscz2.5297x9630111kkuu}
\end{equation}
Now,
employing almost exactly the same arguments as in the proof of Lemma 3.3   in \cite{Zhengsdsd6},
 we conclude that
 the solution of \dref{1.1fghyuisda} satisfies that for all $p>1$,
\begin{equation}
\begin{array}{rl}
&\disp{\frac{1}{{p}}\frac{d}{dt}\|n_{\varepsilon}+\varepsilon\|^{{{p}}}_{L^{{p}}(\Omega)}+
\frac{2m(p-1)}{(m+p-1)^2}\int_{\Omega} |\nabla (n_{\varepsilon}+\varepsilon)^{\frac{m+p-1}{2}}|^2\leq
C_8\int_\Omega (n_{\varepsilon}+\varepsilon)^{p+1-m}|\nabla c_{\varepsilon}|^2 }\\
\end{array}
\label{cz2.5ghhjuyui22ssxccnnihjj}
\end{equation}
for all $t\in(0,T_{max,\varepsilon})$ and some positive constant $C_7.$
By the  H\"{o}lder inequality and \dref{zjccffgbhhjvcc22344789c4456788cvvvbbvscz2.5297x9630111kkuu} and using $m>2$ and the Gagliardo--Nirenberg inequality, we derive there exist positive constants $C_9,C_{10}$ and $C_{11}$ such that
\begin{equation}
\begin{array}{rl}
&\disp\int_\Omega (n_{\varepsilon}+\varepsilon)^{p+1-m} |\nabla c_{\varepsilon}|^2\\\
\leq&\disp{ \left(\int_\Omega{(n_{\varepsilon}+\varepsilon)^{4(p+1-m)}}\right)^{\frac{1}{4}}\left(\int_\Omega |\nabla c_{\varepsilon}|^{\frac{8}{3}}\right)^{\frac{3}{4}}}\\
\leq&\disp{ C_9\|  (n_{\varepsilon}+\varepsilon)^{\frac{p+m-1}{2}}\|^{\frac{2(p+1-m)}{p+m-1}}_{L^{\frac{8(p+1-m)}{p+m-1}}(\Omega)}}\\
\leq&\disp{C_{10}(\|\nabla   (n_{\varepsilon}+\varepsilon)^{\frac{p+m-1}{2}}\|_{L^2(\Omega)}^{\mu_1}\|  (n_{\varepsilon}+\varepsilon)^{\frac{p+m-1}{2}}\|_{L^\frac{2}{p+m-1}(\Omega)}^{1-\mu_1}+\|  (n_{\varepsilon}+\varepsilon)^{\frac{p+m-1}{2}}\|_{L^\frac{2}{p+m-1}(\Omega)})^{\frac{2(p+1-m)}{p+m-1}}}\\
\leq&\disp{C_{11}(\|\nabla   (n_{\varepsilon}+\varepsilon)^{\frac{p+m-1}{2}}\|_{L^2(\Omega)}^{\frac{2(p+1-m)\mu_1}{p+m-1}}+1)}\\
=&\disp{C_{11}(\|\nabla   (n_{\varepsilon}+\varepsilon)^{\frac{p+m-1}{2}}\|_{L^2(\Omega)}^{\frac{12p-12m+9}{6p+6m-8}}+1)~~\mbox{for all}~~ t\in(0, T_{max,\varepsilon}),}\\
\end{array}
\label{cz2.wwsdeddfvgbhnjerfvghyh}
\end{equation}
where
$$\mu_1=\frac{\frac{3[p+m-1]}{2}-\frac{3(p+m-1)}{8(p+1-m)}}{-\frac{1}{2}+\frac{3[p+m-1]}{2}}\in(0,1).$$
Since, $m>2$ yields to $\frac{12p-12m+9}{6p+6m-8}<2$, in light of \dref{cz2.wwsdeddfvgbhnjerfvghyh} and the Young inequality, we derive that
there exists a positive constant $C_{12}$ such that
\begin{equation}
\begin{array}{rl}
\disp C_8\int_\Omega (n_{\varepsilon}+\varepsilon)^{p+1-m} |\nabla c_{\varepsilon}|^2\leq \frac{m(p-1)}{(m+p-1)^2}\|\nabla   (n_{\varepsilon}+\varepsilon)^{\frac{p+m-1}{2}}\|_{L^2(\Omega)}^{\frac{12p-12m+9}{6p+6m-8}}+C_{12}~\mbox{for all}~t\in(0, T_{max,\varepsilon}).\\
\end{array}
\label{cz2.wwddfvgbsdeddfvgbhnjerfvghyh}
\end{equation}
Hence, inserting \dref{cz2.wwddfvgbsdeddfvgbhnjerfvghyh} into \dref{cz2.5ghhjuyui22ssxccnnihjj}, we derive that
\begin{equation}
\begin{array}{rl}
&\disp{\frac{1}{{p}}\frac{d}{dt}\|n_{\varepsilon}+\varepsilon\|^{{{p}}}_{L^{{p}}(\Omega)}+
\frac{2m(p-1)}{(m+p-1)^2}\int_{\Omega} |\nabla (n_{\varepsilon}+\varepsilon)^{\frac{m+p-1}{2}}|^2\leq
C_{12}~\mbox{for all}~t\in(0, T_{max,\varepsilon}). }\\
\end{array}
\label{cz2.5ghhjuyuiwssxxddcc22ssxccnnihjj}
\end{equation}
Now, with some basic analysis, we may derive that for all $p>1,$ there exists a positive constant $C_{13}$ such that
\begin{equation}
\|n_{\varepsilon}(\cdot,t)\|_{L^p(\Omega)}\leq C_{13}~~\mbox{for all}~~ t\in(0, T_{max,\varepsilon}).
\label{ddxcvbbggdddfghhdfddcfvgbbgcz2vv.5gttghju48cfg924ghyuji}
\end{equation}
Let $h_{\varepsilon}(x,t)=\mathcal{P}[-\kappa (Y_{\varepsilon}u_{\varepsilon} \cdot \nabla)u_{\varepsilon}+n_{\varepsilon}\nabla \phi ]$.
Then along with \dref{czfvgb2.5ghhjuyuccvviihjj} and \dref{ddxcvbbggdddfghhdfddcfvgbbgcz2vv.5gttghju48cfg924ghyuji}, there exists a positive constant $C_{13}$ such that $\|h_{\varepsilon}(\cdot,t)\|_{L^2(\Omega)} \leq C_{14}$ for
all $t\in (0, T_{max,\varepsilon})$. Hence,  we pick an arbitrary $\gamma\in (\frac{3}{4}, 1),$ then in light of the  smoothing properties of the
Stokes semigroup (\cite{Giga1215}), we derive  that for some $C_{15} > 0$, we have
\begin{equation}
\begin{array}{rl}
\|A^\gamma u_\varepsilon(\cdot, t)\|_{L^2(\Omega)}\leq&\disp{\|A^\gamma
e^{-tA}u_0\|_{L^2(\Omega)} +\int_0^t\|A^\gamma e^{-(t-\tau)A}h_\varepsilon(\cdot,\tau)d\tau\|_{L^2(\Omega)}d\tau}\\
\leq&\disp{C_{15} t^{-\lambda_1(t-1)}
\|u_0\|_{L^2(\Omega)} +C_{15}\int_0^t(t-\tau)^{-\gamma}\|h_\varepsilon(\cdot,\tau)\|_{L^2(\Omega)}d\tau}\\
\leq&\disp{C_{15}t^{-\lambda_1(t-1)}
\|u_0\|_{L^2(\Omega)} +\frac{C_{14}C_{15}T^{1-\gamma}_{max,\varepsilon}}{1-\gamma}~~ \mbox{for all}~~ t\in(0,T_{max,\varepsilon}).}\\
\end{array}
\label{cz2.57151ccvvhccvvhjjjkkhhggjjllll}
\end{equation}
Observe that $\gamma>\frac{3}{4},$
 $D(A^\gamma)$ is continuously embedded into $L^\infty(\Omega)$, therefore, due to \dref{cz2.57151ccvvhccvvhjjjkkhhggjjllll}, we derive that there exists a positive constant $C_{16}$ such that
 \begin{equation}
\begin{array}{rl}
\|u_\varepsilon(\cdot, t)\|_{L^\infty(\Omega)}\leq  C_{16}~~ \mbox{for all}~~ t\in(0,T_{max,\varepsilon}).\\
\end{array}
\label{cz2.5715jkkcvccvvhjjjkddfffffkhhgll}
\end{equation}
Now, for any $\beta>1$, choosing $p>0$ large enough such that $p>2\beta$, then due to \dref{ddxcvbbggdddfghhdfddcfvgbbgcz2vv.5gttghju48cfg924ghyuji} and \dref{111cz2.5715ssdergghhffghhhtyuftgg1hhkkhhggjjllll}, in light of the Young inequality,  we derive that there exists a positive constant  $C_{17}$ such that
\begin{equation}
\begin{array}{rl}
&\disp{\frac{1}{2\beta}\frac{d}{dt}\int_{\Omega} |\nabla c_{\varepsilon}|^{2\beta} +\frac{1}{2}
\int_{\Omega} |\nabla c_{\varepsilon}|^{2\beta}+\frac{1}{2}\int_\Omega  |\nabla c_{\varepsilon}|^{2\beta-2}|D^2c_{\varepsilon}|^2+
\frac{3(\beta-1)}{8\beta^2}\|\nabla |\nabla c_{\varepsilon}|^\beta\|_{L^2(\Omega)}^2}\\
\leq&\disp{C_{17}~~\mbox{for all}~~ t\in(0, T_{max,\varepsilon}).}\\
\end{array}
\label{111cz2.5715sffggsdergghffgbhhhffghhhtyuftgg1hhkkhhggjjllll}
\end{equation}
Now, integrating the above inequality in time, we derive that there exists a positive constant $C_{18}$ such that
\begin{equation}
\|\nabla c_{\varepsilon}(\cdot,t)\|_{L^{2\beta}(\Omega)}\leq C_{18}~~\mbox{for all}~~ t\in(0, T_{max,\varepsilon})~~~\mbox{and}~~\beta>1.
\label{ddxcvbbggdddfgffgghhdfddcfvgbbgcz2vv.5gttghju48cfg924ghyuji}
\end{equation}

In order to get the boundedness of $\|\nabla c_\varepsilon(\cdot, t)\|_{L^\infty(\Omega)}$,
we rewrite the variation-of-constants formula for $c_{\varepsilon}$ in the form
$$c_\varepsilon(\cdot, t) = e^{t(\Delta-1) }c_0 +\int_0^te^{(t-s)(\Delta-1)}(n_\varepsilon -u_{\varepsilon} \cdot \nabla c_{\varepsilon})(\cdot,s)ds~~ \mbox{for all}~~ t\in(0,T_{max,\varepsilon}).$$
Now, we choose $\theta\in(\frac{7}{8},1),$ then the domain of the fractional power $D((-\Delta + 1)^\theta)\hookrightarrow W^{1,\infty}(\Omega)$
(\cite{Zhangddff4556}). Hence, in view of $L^p$-$L^q$ estimates associated heat semigroup, \dref{ccvvx1.731426677gg}, \dref{ddxcvbbggdddfghhdfddcfvgbbgcz2vv.5gttghju48cfg924ghyuji}, \dref{cz2.5715jkkcvccvvhjjjkddfffffkhhgll} and \dref{ddxcvbbggdddfgffgghhdfddcfvgbbgcz2vv.5gttghju48cfg924ghyuji}, we derive  that there exist positive constants $C_{19},C_{20}$ and $C_{21}$ such that
\begin{equation}
\begin{array}{rl}
&\disp{\|\nabla c_\varepsilon(\cdot, t)\|_{W^{1,\infty}(\Omega)}}\\
\leq&\disp{C_{19}t^{-\theta}e^{-\lambda t}\|c_0\|_{L^{4}(\Omega)}}\\
&+\disp{\int_{0}^t(t-s)^{-\theta}e^{-\lambda(t-s)}\|(n_\varepsilon-u_{\varepsilon} \cdot \nabla c_{\varepsilon})(s)\|_{L^4(\Omega)}ds}\\
\leq&\disp{C_{20}\tau^{-\theta}+C_{20}\int_{0}^t(t-s)^{-\theta}e^{-\lambda(t-s)}+C_{20}\int_{0}^t(t-s)^{-\theta}e^{-\lambda(t-s)}[\|n_\varepsilon(s)\|_{L^4(\Omega)}+
\|\nabla c_{\varepsilon}(s)\|_{L^4(\Omega)}]ds}\\
\leq&\disp{C_{21}~~ \mbox{for all}~~ t\in(\tau,T_{max,\varepsilon})}\\
\end{array}
\label{zjccffgbhjcvvvbscz2.5297x96301ku}
\end{equation}
with $\tau\in(0,T_{max,\varepsilon})$.
Next, using the outcome of \dref{cz2.5ghhjuyui22ssxccnnihjj} with suitably large $p$ as a starting point, we may employ
a Moser-type iteration (see  e.g. Lemma A.1 of  \cite{Tao794}) applied to the first equation of \dref{1.1fghyuisda} to get that
\begin{equation}
\begin{array}{rl}
\|n_{\varepsilon}(\cdot, t)\|_{L^{{\infty}}(\Omega)}\leq C_{22} ~~ \mbox{for all}~~~  t\in(\tau,T_{max,\varepsilon}) \\
\end{array}
\label{cz2.5g5gghh56789hhjui78jj90099}
\end{equation}
and some positive constant $C_{22}$.
In view of \dref{cz2.5715jkkcvccvvhjjjkddfffffkhhgll},  \dref{zjccffgbhjcvvvbscz2.5297x96301ku} and \dref{cz2.5g5gghh56789hhjui78jj90099}, we apply Lemma \ref{lemma70} to reach a contradiction.
\end{proof}

\section{Regularity properties of time derivatives}

In this subsection, we provide some time-derivatives uniform estimates of solutions to the
system \dref{1.1fghyuisda}. The estimate is
used in this Section  to construct the  weak solution of the equation \dref{1.1}.
This will be the purpose of the following three lemma:

\begin{lemma}\label{qqqqlemma45630hhuujjuuyytt}
Let $m>2$,
\dref{dd1.1fghyuisdakkkllljjjkk} and \dref{ccvvx1.731426677gg}
 hold.
 Then for any $T>0, $
  one can find $C > 0$ independent if $\varepsilon$ such that 
\begin{equation}
 \begin{array}{ll}
\disp\int_0^T\|\partial_tn_\varepsilon^{m-1}(\cdot,t)\|_{(W^{2,q}(\Omega))^*}dt  \leq C(T+1)\\
   \end{array}\label{1.1ddfgeddvbnmklllhyuisda}
\end{equation}
as well as
\begin{equation}
 \begin{array}{ll}
  \disp\int_0^T\|\partial_tc_\varepsilon(\cdot,t)\|_{(W^{1,\frac{5}{2}}(\Omega))^*}^{\frac{5}{3}}dt  \leq C(T+1)\\
   \end{array}\label{wwwwwqqqq1.1dddfgbhnjmdfgeddvbnmklllhyussddisda}
\end{equation}
and
\begin{equation}
 \begin{array}{ll}
  \disp\int_0^T\|\partial_tu_\varepsilon(\cdot,t)\|_{(W^{1,2}(\Omega))^*}^2dt  \leq C(T+1).\\
   \end{array}\label{wwwwwqqqq1.1dddfgkkllbhddffgggnjmdfgeddvbnmklllhyussddisda}
\end{equation}
\end{lemma}
\begin{proof}
Firstly,  due to \dref{czfvgb2.5ghhjuyuccvviihjj}, \dref{bnmbncz2.5ghhjuyuivvbnnihjj} and \dref{ddffbnmbnddfgcz2ddfvgbhh.htt678ddfghhhyuiihjj}, employing  the
H\"{o}lder inequality (with two exponents $\frac{4m-1}{4(m-1)}$ and $\frac{4(m-1)}{3}$) and the Gagliardo-Nirenberg inequality, we conclude  that there exist positive
 constants $C_{1},C_2,C_3$ and $C_4$ 
such that
\begin{equation}
\begin{array}{rl}
\disp\int_{0}^T\disp\int_{\Omega}|m(n_{\varepsilon}+\varepsilon)^{m-1}\nabla n_{\varepsilon}|^{\frac{8(m-1)}{4m-1}}
\leq&\disp{C_1\left[\int_{0}^T\disp\int_{\Omega}(n_{\varepsilon}+\varepsilon)^{2m-4}|\nabla n_{\varepsilon}|^2\right]^{\frac{4(m-1)}{4m-1}} \left[\int_{0}^T\disp\int_{\Omega}[n_{\varepsilon}+\varepsilon ]^{\frac{8(m-1)}{3}}\right]^{\frac{3}{4m-1}} }\\
\leq&\disp{C_{2}(T+1)~~\mbox{for all}~~ T > 0}\\
\end{array}
\label{5555ddffbnmbncz2ddfvgffgtyybhh.htt678ghhjjjddfghhhyuiihjj}
\end{equation}
and
\begin{equation}
\begin{array}{rl}
\disp\int_{0}^T\disp\int_{\Omega} |u_{\varepsilon}|^{\frac{10}{3}} =&\disp{\int_{0}^T\| {u_{\varepsilon}}\|^{{\frac{10}{3}}}_{L^{\frac{10}{3}}(\Omega)}}\\
\leq&\disp{C_3\int_{0}^T\left(\| \nabla{u_{\varepsilon}}\|^{2}_{L^{2}(\Omega)}\|{u_{\varepsilon}}\|^{{\frac{4}{3}}}_{L^{2}(\Omega)}+
\|{u_{\varepsilon}}\|^{{\frac{10}{3}}}_{L^{2}(\Omega)}\right)}\\
\leq&\disp{C_4(T+1)~~\mbox{for all}~~ T > 0.}\\
\end{array}
\label{5555bnmbncz2ddfvgffghhbhh.htt678hyuiihjj}
\end{equation}
Next,
testing the first equation of \dref{1.1fghyuisda}
 by certain   $({m-1})n_{\varepsilon}^{m-2}\varphi\in C^{\infty}(\bar{\Omega})$, we have
 \begin{equation}
\begin{array}{rl}
&\disp\left|\int_{\Omega}(n_{\varepsilon}^{m-1})_{t}\varphi\right|\\
 =&\disp{\left|\int_{\Omega}\left[\Delta (n_{\varepsilon}+\varepsilon)^m-\nabla\cdot(n_{\varepsilon}\nabla c_{\varepsilon})-u_{\varepsilon}\cdot\nabla n_{\varepsilon}\right]\cdot({m-1})n_{\varepsilon}^{m-2}\varphi\right|}
\\
\leq&\disp{\left|-(m-1)\int_\Omega \left[m(n_{\varepsilon}+\varepsilon)^{m-1}n_{\varepsilon}^{m-2}\nabla n_{\varepsilon}\cdot\nabla\varphi+(m-2) (n_{\varepsilon}+\varepsilon)^{m-1}n_{\varepsilon}^{m-3}|\nabla n_{\varepsilon}|^2\varphi\right]\right|}\\
&+\disp{(m-1)\left|\int_\Omega[(m-2) n_{\varepsilon}^{m-2}\nabla n_{\varepsilon}\cdot\nabla c_{\varepsilon}\varphi+ n_{\varepsilon}^{m-1}\nabla c_{\varepsilon}\cdot\nabla \varphi]\right|
+\left|\int_\Omega n_{\varepsilon}^{m-1}u_\varepsilon\cdot\nabla\varphi\right|}\\
\leq&\disp{m(m-1)\left\{\int_\Omega \left[(n_{\varepsilon}+\varepsilon)^{m-1}n_{\varepsilon}^{m-2}|\nabla n_{\varepsilon}|+ (n_{\varepsilon}+\varepsilon)^{m-1}n_{\varepsilon}^{m-3}|\nabla n_{\varepsilon}|^2\right]\right\}\|\varphi\|_{W^{1,\infty}(\Omega)}}\\
&+\disp{(m-1)^2\left\{\int_\Omega[ n_{\varepsilon}^{m-2}|\nabla n_{\varepsilon}||\nabla c_{\varepsilon}|+ n_{\varepsilon}^{m-1}|\nabla c_{\varepsilon}|+ n_{\varepsilon}^{m-1}|u_\varepsilon|]\right\}\|\varphi\|_{W^{1,\infty}(\Omega)}}\\
\end{array}
\label{gbhncvbmdcfvgcz2.5ghju48}
\end{equation}
for all $t>0$.
Hence, observe that the embedding $W^{2,q }(\Omega)\hookrightarrow  W^{1,\infty}(\Omega)(q>3)$, due to \dref{ddffbnmbnddfgcz2ddfvgbhh.htt678ddfghhhyuiihjj}, \dref{bnmbncz2.5ghhjuyuivvbnnihjj} and \dref{5555bnmbncz2ddfvgffghhbhh.htt678hyuiihjj}, applying $m>2$ and the Young inequlity, we deduce $C_1,C_2$ and $C_3$ such
that
\begin{equation}
\begin{array}{rl}
&\disp\int_0^T\|\partial_{t}n_{\varepsilon}^{m-1}(\cdot,t)\|_{(W^{2,q }(\Omega))^*}dt\\
\leq&\disp{C_1\left\{\int_0^T\int_{\Omega}(n_{\varepsilon}+\varepsilon)^{2m-4}|\nabla n_{\varepsilon}|^{2}+\int_0^T\int_{\Omega}n_{\varepsilon}^{2m-2}+\int_0^T\int_{\Omega}|\nabla c_{\varepsilon}|^{2}+\int_0^T\int_{\Omega}|u_\varepsilon|^{2}\right\}
}\\
\leq&\disp{C_2\left\{\int_0^T\int_{\Omega}(n_{\varepsilon}+\varepsilon)^{2m-4}|\nabla n_{\varepsilon}|^{2}+\int_0^T\int_{\Omega}|\nabla c_{\varepsilon}|^{2}+\int_0^T\int_{\Omega}n_{\varepsilon}^{\frac{8(m-1)}{3}}+\int_0^T\int_{\Omega}|u_\varepsilon|^{\frac{10}{3}}+T\right\}
}\\
\leq&\disp{C_3(T+1)~~\mbox{for all}~~ T > 0,
}\\
\end{array}
\label{yyygbhncvbmdcxxcdfvgbfvgcz2.5ghju48}
\end{equation}
which implies  \dref{1.1ddfgeddvbnmklllhyuisda}.

Likewise, given any $\varphi\in  C^\infty(\bar{\Omega})$, we may test the second equation in \dref{1.1fghyuisda} against
$\varphi$
to conclude  that
\begin{equation}
\begin{array}{rl}
\disp\left|\int_{\Omega}\partial_{t}c_{\varepsilon}(\cdot,t)\varphi\right|=&\disp{\left|\int_{\Omega}\left[\Delta c_{\varepsilon}-c_{\varepsilon}+n_{\varepsilon}-u_{\varepsilon}\cdot\nabla c_{\varepsilon}\right]\cdot\varphi\right|}
\\
=&\disp{\left|-\int_\Omega \nabla c_{\varepsilon}\cdot\nabla\varphi-\int_\Omega  c_{\varepsilon}\varphi+\int_\Omega n_{\varepsilon} \varphi+\int_\Omega c_{\varepsilon}u_\varepsilon\cdot\nabla\varphi\right|}\\
\leq&\disp{\left\{\|\nabla c_{\varepsilon}\|_{L^{{\frac{5}{3}}}(\Omega)}+\|c_{\varepsilon} \|_{L^{\frac{5}{3}}(\Omega)}+\|n_{\varepsilon} \|_{L^{\frac{5}{3}}(\Omega)}+\|c_{\varepsilon}u_\varepsilon\|_{L^{\frac{5}{3}}(\Omega)}\right\}\|\varphi\|_{W^{1,\frac{5}{2}}(\Omega)}
~~\mbox{for all}~~ t>0.}\\
\end{array}
\label{wwwwwqqqqgbhncvbmdcfxxdcvbccvbbvgcz2.5ghju48}
\end{equation}
Thus, due to
\dref{bnmbncz2.5ghhjuyuivvbnnihjj}, \dref{bnmbncz2.ffghh5ghhjuyuivvbnnihjj}--\dref{ddffbnmbnddfgcz2ddfvgbhh.htt678ddfghhhyuiihjj} and \dref{5555bnmbncz2ddfvgffghhbhh.htt678hyuiihjj}, in light of $m>2$ and the Young inequality,  we derive   that there exist positive constant $C_{8}$ and $C_{9}$ such that
\begin{equation}
\begin{array}{rl}
&\disp\int_0^T\|\partial_{t}c_{\varepsilon}(\cdot,t)\|^{\frac{5}{3}}_{(W^{1,\frac{5}{2}}(\Omega))^*}dt\\
\leq&\disp{C_{8}\left(\int_0^T\int_\Omega|\nabla c_{\varepsilon}|^{2}+\int_0^T\int_\Omega n_{\varepsilon}^\frac{8(m-1)}{3}+\int_0^T\int_\Omega c_\varepsilon^{\frac{40(m-1)}{9}}+\int_0^T\int_\Omega |u_\varepsilon|^{\frac{10}{3}}+T\right)}\\
\leq&\disp{C_{9}(T+1)
~~\mbox{for all}~~ T>0.}\\
\end{array}
\label{wwwwwqqqqgbhncvbmdcfxxxcvbddfghnxdcvbbbvgcz2.5ghju48}
\end{equation}
Hence, \dref{wwwwwqqqq1.1dddfgbhnjmdfgeddvbnmklllhyussddisda} holds.

Finally, for any given  $\varphi\in C^{\infty}_{0,\sigma} (\Omega;\mathbb{R}^3)$, we infer from the third equation in \dref{1.1fghyuisda} that
\begin{equation}
\begin{array}{rl}
\disp\left|\int_{\Omega}\partial_{t}u_{\varepsilon}(\cdot,t)\varphi\right|=&\disp{\left|-\int_\Omega \nabla u_{\varepsilon}\cdot\nabla\varphi-\kappa\int_\Omega (Y_{\varepsilon}u_{\varepsilon}\otimes u_{\varepsilon})\cdot\nabla\varphi+\int_\Omega n_{\varepsilon}\nabla \phi\cdot\varphi\right|
~~\mbox{for all}~~ t>0.}\\
\end{array}
\label{wwwwwqqqqgbhncvbmdcfxxdcvbccvqqwerrbbvgcz2.5ghju48}
\end{equation}
 Now,   by \dref{bnmbncz2.5ghhjuyuivvbnnihjj}, \dref{bnmbncz2.ffghh5ghhjuyuivvbnnihjj} and \dref{ssdcfvgdhhjjdfghgghjjnnhhkklld911cz2.5ghju48}, we also get  that there exist positive constants
 $C_{10},C_{11}$ and $C_{12}$ such that
\begin{equation}
\begin{array}{rl}
&\disp\int_0^T\|\partial_{t}u_{\varepsilon}(\cdot,t)\|^{2}_{(W^{1,2}(\Omega))^*}dt\\
\leq&\disp{C_{10}\left(\int_0^T\int_\Omega|\nabla u_{\varepsilon}|^{2}+\int_0^T\int_\Omega |Y_{\varepsilon}u_{\varepsilon}\otimes u_{\varepsilon}|^{2}+\int_0^T\int_\Omega n_\varepsilon^{2}\right)}\\
\leq&\disp{C_{11}\left(\int_0^T\int_\Omega|\nabla u_{\varepsilon}|^{2}+\int_0^T\int_\Omega |Y_{\varepsilon}u_\varepsilon|^{2}+\int_0^T\int_\Omega n_{\varepsilon}^{\frac{8(m-1)}{3}}+T\right)}\\
\leq&\disp{C_{12}(T+1)
~~\mbox{for all}~~ T>0.}\\
\end{array}
\label{wwwwwqqqqgbhncvbmdcfxxxcvxxcvvbddfghnxdddffckkvbgtyyiiobddfffbvgcz2.5ghju48}
\end{equation}
Hence, \dref{wwwwwqqqq1.1dddfgkkllbhddffgggnjmdfgeddvbnmklllhyussddisda} is hold.
\end{proof}

In order to prove
the limit functions $n$ and $c$ gained below, we will rely on an
additional regularity estimate for %
$u_\varepsilon\cdot\nabla c_\varepsilon$,
$n_{\varepsilon}\nabla c_{\varepsilon}$ and $n_\varepsilon u_{\varepsilon}$.

\begin{lemma}\label{4455lemma45630hhuujjuuyytt}
Let $m>2$,
\dref{dd1.1fghyuisdakkkllljjjkk} and \dref{ccvvx1.731426677gg}
 hold.
 Then for any $T>0, $
  one can find $C > 0$ independent of $\varepsilon$ such that 
\begin{equation}
 \begin{array}{ll}
  \disp\int_0^T\int_{\Omega}|n_\varepsilon \nabla c_\varepsilon|^{\frac{8(m-1)}{4m-1}} \leq C(T+1)\\
   \end{array}\label{1.1dddfgbhnjmdfgeddvbnmklllhyussddisda}
\end{equation}
and
\begin{equation}
 \begin{array}{ll}
  \disp\int_0^T\int_{\Omega}|u_\varepsilon\cdot\nabla c_\varepsilon|^{\frac{5}{4}}   \leq C(T+1).\\
   \end{array}\label{1.1dddfgbhnjmdfgeddvbnmklllhyussdddfgggdisda}
\end{equation}
\end{lemma}
\begin{proof}
In light of
 \dref{bnmbncz2.5ghhjuyuivvbnnihjj}, \dref{ddffbnmbnddfgcz2ddfvgbhh.htt678ddfghhhyuiihjj}, \dref{5555bnmbncz2ddfvgffghhbhh.htt678hyuiihjj}
   and   the Young inequality, we derive that  there exist positive
constants $C_{1}$ and $C_2$ such that 
\begin{equation}
\begin{array}{rl}
\disp\int_{0}^T\disp\int_{\Omega}|n_{\varepsilon}\nabla c_{\varepsilon}|^{\frac{8(m-1)}{4m-1}}
\leq&\disp{\left(\int_0^T\int_{\Omega}|\nabla c_{\varepsilon}|^{2}\right)^{\frac{3}{4m-1}}\left(\int_0^T\int_{\Omega}n_{\varepsilon}^{\frac{8(m-1)}{3}}\right)^{\frac{4(m-1)}{4m-1}}}\\
\leq&\disp{C_{1}(T+1)~~\mbox{for all}~~ T > 0}\\
\end{array}
\label{ddfff5555ddffbnmbncz2ddfvgffgtyybhh.htt678ghhjjjddfghhhyuiihjj}
\end{equation}
and
\begin{equation}
\begin{array}{rl}
\disp\int_{0}^T\disp\int_{\Omega}|u_\varepsilon\cdot\nabla c_\varepsilon|^{\frac{5}{4}}
\leq&\disp{\left(\int_0^T\int_{\Omega}|\nabla c_{\varepsilon}|^{2}\right)^{\frac{5}{8}}
\left(\int_0^T\int_{\Omega}|u_{\varepsilon}|^{\frac{10}{3}}\right)^{\frac{3}{8}}}\\
\leq&\disp{C_{2}(T+1)~~\mbox{for all}~~ T > 0.}\\
\end{array}
\label{ddfff5555ddffbnmbncz2ddfvgffgddfftyybhh.httff678ghhjjjddfghhhyuiihjj}
\end{equation}
These  readily establish \dref{1.1dddfgbhnjmdfgeddvbnmklllhyussddisda} and \dref{1.1dddfgbhnjmdfgeddvbnmklllhyussdddfgggdisda}.

\end{proof}


%
%

\section{Passing to the limit. Proof of Theorem  \ref{theorem3}}
With the above compactness properties at hand, by means of a standard extraction procedure we can
now derive the following lemma which actually contains our main existence result already.

%
%
%
{\bf The proof of Theorem  \ref{theorem3}}
Firstly, 
in light of Lemmata \ref{lemmaghjssddgghhmk4563025xxhjklojjkkk}--\ref{lemmddaghjssddgghhmk4563025xxhjklojjkkk} and \ref{qqqqlemma45630hhuujjuuyytt},
we conclude that there exists a positive constant $C_1$  such that
\begin{equation}
\begin{array}{rl}
\|n_{\varepsilon}^{m-1}\|_{L^{2}_{loc}([0,\infty); W^{1,2}(\Omega))}\leq C_1(T+1)~~~\mbox{and}~~~\|\partial_{t}n_{\varepsilon}^{m-1}\|_{L^{1}_{loc}([0,\infty); (W^{2,q}(\Omega))^*)}\leq C_2(T+1)
\end{array}
\label{ggjjssdffzcddffcfccvvfggvvvvgccvvvgjscz2.5297x963ccvbb111kkuu}
\end{equation}
as well as
\begin{equation}
\begin{array}{rl}
\|c_{\varepsilon}\|_{L^{2}_{loc}([0,\infty); W^{1,2}(\Omega))}\leq C_1(T+1)~~~\mbox{and}~~~\|\partial_{t}c_{\varepsilon}\|_{L^{1}_{loc}([0,\infty); (W^{1,\frac{5}{2}}(\Omega)))^*)}\leq C_1(T+1)
\end{array}
\label{ggjjssdffzcddffcfccvvfggvvddffvvgccvvvgjscz2.5297x963ccvbb111kkuu}
\end{equation}
and
\begin{equation}
\begin{array}{rl}
\|u_{\varepsilon}\|_{L^{2}_{loc}([0,\infty); W^{1,2}(\Omega))}\leq C_1(T+1)~~~\mbox{and}~~~\|\partial_{t}u_{\varepsilon}\|_{L^{1}_{loc}([0,\infty); (W^{1,2}(\Omega))^*)}\leq C_1(T+1).
\end{array}
\label{ggjjssdffzcddffcfccvvfggvvvvgccffghhvvvgjscz2.5297x963ccvbb111kkuu}
\end{equation}
Hence, collecting
\dref{ggjjssdffzcddffcfccvvfggvvddffvvgccvvvgjscz2.5297x963ccvbb111kkuu}--\dref{ggjjssdffzcddffcfccvvfggvvvvgccffghhvvvgjscz2.5297x963ccvbb111kkuu} and  employing the the Aubin-Lions lemma (see e.g.
\cite{Simon}), we conclude that
\begin{equation}
\begin{array}{rl}
(c_{\varepsilon})_{\varepsilon\in(0,1)}~~~\mbox{is strongly precompact in}~~~L^{2}_{loc}(\bar{\Omega}\times[0,\infty))
\end{array}
\label{ggjjssdffzcddddffcfcchhvvfggvvddffvvgccvvvgjscz2.5297x963ccvbb111kkuu}
\end{equation}
and
\begin{equation}
\begin{array}{rl}
(u_{\varepsilon})_{\varepsilon\in(0,1)}~~~\mbox{is strongly precompact in}~~~L^{2}_{loc}(\bar{\Omega}\times[0,\infty)).
\end{array}
\label{ggjjssdffzcccddffcfccvvfggvvggvvgccffghhvvvgjscz2.5297x963ccvbb111kkuu}
\end{equation}
Therefore, there exists a subsequence $\varepsilon=\varepsilon_j\subset (0,1)_{j\in \mathbb{N}}$
and the limit functions $n$ and $c$
such that
\begin{equation}
c_\varepsilon\rightarrow c ~~\mbox{in}~~ L^{2}_{loc}(\bar{\Omega}\times[0,\infty))~~\mbox{and}~~\mbox{a.e.}~~\mbox{in}~~\Omega\times(0,\infty),
 \label{zjscz2.5297x963ddfgh06662222tt3}
\end{equation}
\begin{equation}
u_\varepsilon\rightarrow u~~\mbox{in}~~ L_{loc}^2(\bar{\Omega}\times[0,\infty))~~\mbox{and}~~\mbox{a.e.}~~\mbox{in}~~\Omega\times(0,\infty),
 \label{zjscz2.5297x96302222t666t4}
\end{equation}
\begin{equation}
\nabla c_\varepsilon\rightharpoonup \nabla c~~\begin{array}{ll}
 \mbox{in}~~ L_{loc}^{2}(\bar{\Omega}\times[0,\infty)),
   \end{array}\label{1.1ddfgghhhge666ccdf2345ddvbnmklllhyuisda}
\end{equation}
and
\begin{equation}
 \nabla u_\varepsilon\rightharpoonup \nabla u ~~\mbox{ in}~~L^{2}_{loc}(\bar{\Omega}\times[0,\infty)).
 \label{zjscz2.5297x96366602222tt4455}
\end{equation}
Next, in view of \dref{ggjjssdffzcddffcfccvvfggvvvvgccvvvgjscz2.5297x963ccvbb111kkuu}, an Aubin--Lions lemma (see e.g.
\cite{Simon}) applies to yield strong precompactness of $(n_\varepsilon^{m-1})_{\varepsilon\in(0,1)}$ in
$L^2(\Omega\times(0,T)),$ whence along a suitable subsequence we may derive that $n_\varepsilon^{m-1}\rightarrow z^{m-1}_1$
 and hence $n_{\varepsilon}\rightarrow z_1$ a.e. in $\Omega\times(0,\infty)$ for some nonnegative measurable $z_1 : \Omega\times(0,\infty)\rightarrow \mathbb{R}$.
  Now, with the help of the Egorov
theorem, we conclude  that necessarily $z_1 = n,$ thus
\begin{equation} n_\varepsilon\rightarrow n ~~\mbox{a.e.}~~ \mbox{in}~~ \Omega\times (0,\infty).
\label{z666jscz2.5297xddrfff9630222222}
\end{equation}
Therefore, observing that ${\frac{8(m-1)}{4m-1}}>1,{\frac{8(m-1)}{3}}>1,$ due to \dref{5555ddffbnmbncz2ddfvgffgtyybhh.htt678ghhjjjddfghhhyuiihjj}--\dref{5555bnmbncz2ddfvgffghhbhh.htt678hyuiihjj}, \dref{ddffbnmbnddfgcz2ddfvgbhh.htt678ddfghhhyuiihjj},
there exists a subsequence $\varepsilon=\varepsilon_j\subset (0,1)_{j\in \mathbb{N}}$
such that
$\varepsilon_j\searrow 0$ as $j \rightarrow\infty$ 
\begin{equation}
(n_{\varepsilon}+\varepsilon)^{m-1}\nabla n_{\varepsilon}\rightharpoonup n^{m-1}\nabla n~~\begin{array}{ll}
\mbox{in}~~~L_{loc}^{\frac{8(m-1)}{4m-1}}(\bar{\Omega}\times[0,\infty))\\
   \end{array}\label{1.1666ddccvvfggfgghhhgeccdf2345ddvbnmklllhyuisda}
\end{equation}
as well as
\begin{equation}
u_\varepsilon\rightharpoonup u ~~\mbox{in}~~ L^{\frac{10}{3}}_{loc}(\bar{\Omega}\times[0,\infty))
 \label{zjscz2.5297x66696302222tt4}
\end{equation}
and
\begin{equation}
 n_\varepsilon\rightharpoonup n~~ \begin{array}{ll}
 \mbox{in}~~ L_{loc}^{\frac{8(m-1)}{3}}(\bar{\Omega}\times[0,\infty)).
   \end{array}\label{1.1666ddfgffgghheccdf2345ddvbnmklllhyuisda}
\end{equation}
Next, let $g_\varepsilon(x, t) := -c_\varepsilon+n_{\varepsilon}-u_{\varepsilon}\cdot\nabla c_{\varepsilon}.$
Therefore, recalling \dref{ddffbnmbnddfgcz2ddfvgbhh.htt678ddfghhhyuiihjj}, \dref{bnmbncz2.5ghhjuyuivvbnnihjj} and \dref{1.1dddfgbhnjmdfgeddvbnmklllhyussdddfgggdisda}, we conclude that $c_{\varepsilon t}-c_{\varepsilon } = g_\varepsilon$
is bounded in $L^{\frac{5}{4}} (\Omega\times(0, T))$ for any  $\varepsilon\in(0,1)$,  we may invoke the standard parabolic regularity theory  to infer that
$(c_{\varepsilon})_{\varepsilon\in(0,1)}$ is bounded in
$L^{\frac{5}{4}} ((0, T); W^{2,\frac{5}{4}}(\Omega))$.
Thus,  by \dref{wwwwwqqqq1.1dddfgbhnjmdfgeddvbnmklllhyussddisda} and the Aubin--Lions lemma we derive that  the relative compactness of $(c_{\varepsilon})_{\varepsilon\in(0,1)}$ in
$L^{\frac{5}{4}} ((0, T); W^{1,\frac{5}{4}}(\Omega))$. We can pick an appropriate subsequence which is
still written as $(\varepsilon_j )_{j\in \mathbb{N}}$ such that $\nabla c_{\varepsilon_j} \rightarrow z_2$
 in $L^{\frac{5}{4}} (\Omega\times(0, T))$ for all $T\in(0, \infty)$ and some
$z_2\in L^{\frac{5}{4}} (\Omega\times(0, T))$ as $j\rightarrow\infty$, hence $\nabla c_{\varepsilon_j} \rightarrow z_2$ a.e. in $\Omega\times(0, \infty)$
 as $j \rightarrow\infty$.
In view  of \dref{1.1ddfgghhhge666ccdf2345ddvbnmklllhyuisda} and  the Egorov theorem we conclude  that
$z_2=\nabla c,$ and whence
\begin{equation}
\nabla c_\varepsilon\rightarrow \nabla c~~\begin{array}{ll}
  ~\mbox{a.e.}~~\mbox{in}~~\Omega\times(0,\infty)~~~\mbox{as}~~\varepsilon =\varepsilon_j\searrow0.
   \end{array}\label{1.1ddhhyujiiifgghhhge666ccdf2345ddvbnmklllhyuisda}
\end{equation}
In the following, we shall prove $(n,c,u)$ is a weak solution of problem \dref{1.1} in Definition \ref{df1}.
In fact, with the help of  \dref{zjscz2.5297x963ddfgh06662222tt3}--\dref{zjscz2.5297x96366602222tt4455}, \dref{1.1666ddfgffgghheccdf2345ddvbnmklllhyuisda}, we can derive  \dref{dffff1.1fghyuisdakkklll}.
Now, by the nonnegativity of $n_\varepsilon$ and $c_\varepsilon$, we derive  $n \geq 0$ and $c\geq 0$. Next, due to
\dref{zjscz2.5297x96366602222tt4455} and $\nabla\cdot u_{\varepsilon} = 0$, we conclude that
$\nabla\cdot u = 0$ a.e. in $\Omega\times (0, \infty)$.
On the other hand, in view of \dref{bnmbncz2.5ghhjuyuivvbnnihjj} and \dref{ddffbnmbnddfgcz2ddfvgbhh.htt678ddfghhhyuiihjj}, we can infer from
\dref{1.1dddfgbhnjmdfgeddvbnmklllhyussddisda} that
$$
n_{\varepsilon}\nabla c_\varepsilon\rightharpoonup z_3~~\begin{array}{ll}
  ~~~\mbox{in}~~ L^{\frac{8(m-1)}{3}}(\Omega\times(0,T))~~\mbox{for each}~~ T \in(0, \infty).
   \end{array}
   $$
Next, due to \dref{zjscz2.5297x963ddfgh06662222tt3}, \dref{z666jscz2.5297xddrfff9630222222} and \dref{1.1ddhhyujiiifgghhhge666ccdf2345ddvbnmklllhyuisda}, we derive that
\begin{equation}n_\varepsilon\nabla c_\varepsilon\rightarrow n\nabla c~~~\mbox{a.e.}~~\mbox{in}~~\Omega\times(0,\infty)~~~\mbox{as}~~\varepsilon =\varepsilon_j\searrow0.
\label{1.1ddddfddffttyygghhyujiiifgghhhgffgge666ccdf2345ddvbnmklllhyuisda}
\end{equation}
Therefore, by the Egorov theorem, we can get $z_3= n\nabla c,$ and hence
\begin{equation}
n_{\varepsilon}\nabla c_\varepsilon\rightharpoonup n\nabla c~~\begin{array}{ll}
  ~~~\mbox{in}~~ L^{\frac{8(m-1)}{3}}(\Omega\times(0,T))~~\mbox{for each}~~ T \in(0, \infty).
   \end{array}\label{zxxcvvfgggjscddf4556ffg77ffcvvfggz2.5ty}
\end{equation}
Next, due to  $\frac{3}{8(m-1)}+\frac{3}{10}<\frac{3}{4}$, in view of  \dref{zjscz2.5297x66696302222tt4} and \dref{1.1666ddfgffgghheccdf2345ddvbnmklllhyuisda},
%
 we also  infer that for each $T\in(0, \infty)$
$$
n_\varepsilon u_\varepsilon\rightharpoonup z_4 ~~\mbox{ in}~~ L^{\frac{4}{3}}(\Omega\times(0,T))~~~\mbox{as}~~\varepsilon=\varepsilon_j\searrow0,
 $$
and moreover, \dref{zjscz2.5297x96302222t666t4} and \dref{z666jscz2.5297xddrfff9630222222} imply that
\begin{equation}
n_\varepsilon u_\varepsilon\rightarrow nu ~~\mbox{a.e.}~~\mbox{in}~~\Omega\times(0,\infty)~~~\mbox{as}~~\varepsilon =\varepsilon_j\searrow0,
 \label{zxxcvvfgggjscfffgyhhfggz2.529ddfgg7x963022ccd22tt4}
\end{equation}
which combined with the Egorov theorem 
 implies that
%
%
\begin{equation}
n_\varepsilon u_\varepsilon\rightharpoonup nu ~~\mbox{ in}~~ L^{\frac{4}{3}}(\Omega\times(0,T))~~~\mbox{as}~~\varepsilon=\varepsilon_j\searrow0
 \label{zxxcvvfgggjscddfffcvvfggz2.5297ffghhx96302222tt4}
\end{equation}
for each $T \in(0, \infty).$
As a straightforward consequence of \dref{zjscz2.5297x963ddfgh06662222tt3} and \dref{zjscz2.5297x96302222t666t4}, it holds that
\begin{equation}
c_\varepsilon u_\varepsilon\rightarrow cu ~~\mbox{ in}~~ L^{1}_{loc}(\bar{\Omega}\times(0,\infty))~~~\mbox{as}~~\varepsilon=\varepsilon_j\searrow0.
 \label{zxxcvvfgggjscddfffcvvfggz2.5297x96302222tt4}
\end{equation}

%
%
%
%
 Next, by  \dref{zjscz2.5297x96302222t666t4} and
using the fact that
 $\|Y_{\varepsilon}\varphi\|_{L^2(\Omega)} \leq \|\varphi\|_{L^2(\Omega)}(\varphi\in L^2_{\sigma}(\Omega))$
and
$Y_{\varepsilon}\varphi \rightarrow \varphi$ in $L^2(\Omega)$ as $\varepsilon\searrow0$, we derive that there exists a positive constant $C_2$ such that
\begin{equation}
\begin{array}{rl}
\left\|Y_{\varepsilon}u_{\varepsilon}(\cdot,t)-u(\cdot,t)\right\|_{L^2(\Omega)}  \leq&\disp{\left\|Y_{\varepsilon}[u_{\varepsilon}(\cdot,t)-u(\cdot,t)]\right\|_{L^2(\Omega)}+
\left\|Y_{\varepsilon}u(\cdot,t)-u(\cdot,t)\right\|_{L^2(\Omega)}}\\
\leq&\disp{\left\|u_{\varepsilon}(\cdot,t)-u(\cdot,t)\right\|_{L^2(\Omega)}+
\left\|Y_{\varepsilon}u(\cdot,t)-u(\cdot,t)\right\|_{L^2(\Omega)}}\\
\rightarrow&\disp{0~~\mbox{as}~~\varepsilon=\varepsilon_j\searrow0}\\
\end{array}
\label{ggjjssdffzccvvvvggjscz2.5297x963ccvbb111kkuu}
\end{equation}
and
\begin{equation}
\begin{array}{rl}
\left\|Y_{\varepsilon}u_{\varepsilon}(\cdot,t)-u(\cdot,t)\right\|_{L^2(\Omega)}^2  \leq&\disp{\left(\|Y_{\varepsilon}u_{\varepsilon}(\cdot,t)|\|_{L^2(\Omega)}+\|u(\cdot,t)|\|_{L^2(\Omega)}\right)^2}\\
\leq&\disp{\left(\|u_{\varepsilon}(\cdot,t)|\|_{L^2(\Omega)}+\|u(\cdot,t)|\|_{L^2(\Omega)}\right)^2}\\
\leq&\disp{C_2~~\mbox{for all}~~t\in(0,\infty)~~\mbox{and}~~\varepsilon\in(0,1). }\\
\end{array}
\label{ggjjssdffzccffggvvvvggjscz2.5297x963ccvbb111kkuu}
\end{equation}
Now, thus, by \dref{zjscz2.5297x96302222t666t4}, \dref{ggjjssdffzccvvvvggjscz2.5297x963ccvbb111kkuu} and \dref{ggjjssdffzccffggvvvvggjscz2.5297x963ccvbb111kkuu} and the dominated convergence theorem, we derive that
\begin{equation}
\begin{array}{rl}
\disp\int_{0}^T\|Y_{\varepsilon}u_{\varepsilon}(\cdot,t)-u(\cdot,t)\|_{L^2(\Omega)}^2dt\rightarrow0 ~~\mbox{as}~~\varepsilon=\varepsilon_j\searrow0 ~~~\mbox{for all}~~T>0.
\end{array}
\label{ggjjssdffzccffggvvvvgccvvvgjscz2.5297x963ccvbb111kkuu}
\end{equation}
which implies that
\begin{equation}
Y_\varepsilon u_\varepsilon\rightarrow u ~~\mbox{in}~~ L_{loc}^2([0,\infty); L^2(\Omega)).
 \label{zjscz2.5297x96302266622tt44}
\end{equation}
Now, collecting \dref{zjscz2.5297x96302222t666t4}  and \dref{zjscz2.5297x96302266622tt44}, we derive
\begin{equation}
\begin{array}{rl}
Y_{\varepsilon}u_{\varepsilon}\otimes u_{\varepsilon}\rightarrow u \otimes u ~~\mbox{in}~~L^1_{loc}(\bar{\Omega}\times[0,\infty))~~\mbox{as}~~\varepsilon=\varepsilon_j\searrow0.
\end{array}
\label{ggjjssdffzccfccvvfgghjjjvvvvgccvvvgjscz2.5297x963ccvbb111kkuu}
\end{equation}
Therefore, by \dref{zxxcvvfgggjscddf4556ffg77ffcvvfggz2.5ty}--\dref{zxxcvvfgggjscddfffcvvfggz2.5297x96302222tt4} and \dref{ggjjssdffzccfccvvfgghjjjvvvvgccvvvgjscz2.5297x963ccvbb111kkuu} we conclude that the integrability of
$n\nabla c, nu$ and $cu,u\otimes u$ in \dref{726291hh}.
 Finally, according to \dref{zjscz2.5297x963ddfgh06662222tt3}--\dref{zxxcvvfgggjscddfffcvvfggz2.5297x96302222tt4} and \dref{zjscz2.5297x96302266622tt44}--\dref{ggjjssdffzccfccvvfgghjjjvvvvgccvvvgjscz2.5297x963ccvbb111kkuu}, we may pass to the limit in
the respective weak formulations associated with the the regularized system \dref{1.1fghyuisda} and get
 the integral
identities \dref{eqx45xx12112ccgghh}--\dref{eqx45xx12112ccgghhjjgghh}.

{\bf Acknowledgement}:
 The authors are very grateful to the anonymous reviewers for their carefully reading and valuable suggestions
which greatly improved this work.
This work is partially supported by  the National Natural
Science Foundation of China (No. 11601215),
the
Natural Science Foundation of Shandong Province of China (No. ZR2016AQ17) and the Doctor Start-up Funding of Ludong University (No. LA2016006).

\end{document}